\documentclass{article}

\usepackage[T1]{fontenc}
\usepackage[utf8]{inputenc}

\usepackage{amsmath}
\usepackage{amssymb}

\usepackage[
    backend=biber,
    style=numeric,
    sorting=none
]{biblatex}

\usepackage{graphicx}
\usepackage{subcaption}

\usepackage[percent]{overpic}
\usepackage{xcolor}
\definecolor{vone}{RGB}{41,92,158}       
\definecolor{vtwo}{RGB}{115,166,214}     
\definecolor{varith}{RGB}{209,122,46}    
\definecolor{vharm}{RGB}{237,179,115}    

\usepackage{subcaption}

\usepackage{hyperref}
\usepackage{cleveref}

\usepackage[a4paper,margin=1.1in]{geometry}

\usepackage{amsthm}

\newtheorem{theorem}{Theorem}[section]
\newtheorem{proposition}[theorem]{Proposition}
\newtheorem{lemma}[theorem]{Lemma}

\theoremstyle{definition}

\theoremstyle{remark}

\usepackage{booktabs} 

\usepackage{algorithm}
\usepackage{algorithmic}

\title{Fréchet Means of Periodic Orbits in Dynamical Systems:
Geometry, Dynamics, and Diagnostics}

\author{
Christian Kuehn\thanks{Technical University of Munich, Germany.
\href{mailto:christian.kuehn@tum.de}{\texttt{christian.kuehn@tum.de}}}
\and Paulina Sophia Hering\thanks{Technical University of Munich, Germany.
Supported by the TUM International Graduate School of Science and Engineering (IGSSE).
Corresponding author. \href{mailto:paulina.hering@tum.de}{\texttt{paulina.hering@tum.de}}
}}

\date{}

\begin{document}

\maketitle
\begin{abstract}
In many dynamical systems applications, one seeks representative trajectories summarizing families of periodic orbits arising from parameter variation or uncertainty. While the Fréchet mean provides a natural notion of averaging in nonlinear metric spaces, its application to periodic trajectories is complicated by the interplay between geometric shape and temporal dynamics. Inspired by ideas from shape analysis, we develop a framework for computing Fréchet means of periodic orbits by representing trajectories as closed curves and introducing a metric structure on a quotient space that accounts for circular phase shifts. Within this framework, we establish the existence of empirical Fréchet means in the resulting infinite-dimensional quotient space. A central finding is that the resulting Fréchet mean depends strongly on the chosen parametrization: time parametrization preserves dynamical information, whereas arc length parametrization emphasizes geometric structure. To reconcile these viewpoints, we propose a decoupled approach that computes a geometric Fréchet mean using arc length parametrization and subsequently reconstructs representative dynamics through harmonic averaging of aligned speed profiles. We further introduce curvature- and medoid-based diagnostic measures that quantify the representativeness of the resulting mean and identify situations in which averaging produces geometric artifacts or fails to capture heterogeneous ensembles. Numerical experiments for the Van der Pol oscillator, the Rosenzweig--MacArthur predator--prey model, and the Morris--Lecar neuronal model demonstrate that the proposed methodology yields robust geometric summaries together with consistent averaged dynamics. The framework provides a principled approach for constructing representative periodic trajectories and assessing their validity in uncertainty quantification and dynamical systems.
\end{abstract}

\section{Introduction}

Dynamical systems provide a mathematical framework for describing the evolution of states over time, typically through differential equations or discrete-time maps. A central goal in their analysis is to understand the qualitative behavior of solutions, which is often governed by invariant objects such as equilibria, periodic orbits, and more general invariant sets. These objects capture the long-term dynamics of the system and arise in applications ranging from climate science and ecology to neuroscience and engineering, see, for example, \cite{strogatz2024nonlinear, guckenheimer2013nonlinear, kuznetsov1998elements}. In many applications, however, invariant objects vary with parameters, initial conditions, or modeling assumptions. This naturally leads to the study of ensembles of invariant objects and raises the question of how such ensembles can be summarized by representative quantities. Such summaries provide compact descriptions of complex datasets, facilitate comparisons across parameter regimes, and can serve as representative dynamical signatures for visualization and analysis.

Uncertainty quantification (UQ) provides a framework for studying the impact of parameter variability and stochasticity on dynamical systems. Methods such as Monte Carlo sampling, sensitivity analysis, polynomial chaos expansions, and Bayesian inference enable the characterization of uncertainty in model outputs and dynamical quantities~\cite{sullivan2015introduction, smith2024uncertainty, le2010spectral, latifi2025survey}. More recently, polynomial chaos techniques have been extended to the computation of invariant objects under parameter uncertainty, including periodic orbits~\cite{breden2020computing,Breden} and bifurcations~\cite{KuehnLux,Luxetal,KuehnPiazzolaUllmann}. These developments allow ensembles of periodic solutions to be computed efficiently, but they do not address how such ensembles can be represented by a single representative trajectory.

For equilibria, this problem is straightforward: an ensemble can be summarized by the arithmetic mean of its locations. For periodic solutions, however, the situation is fundamentally more challenging. Parameter variations often generate families of limit cycles rather than isolated equilibria, as encountered in oscillatory climate phenomena, predator--prey dynamics, and neuronal models~\cite{jin1997equatorial, rosenzweig1963graphical, morris1981voltage}. Unlike equilibria, periodic trajectories are geometric objects whose representation depends on parametrization, phase alignment, and other symmetries. Consequently, averaging periodic orbits requires more than a pointwise averaging procedure. Naive averaging can distort essential geometric and dynamical features and may produce trajectories that are not representative of any observed behavior~\cite{ramsay1997functional, srivastava2016functional}.

A natural extension of averaging to nonlinear spaces is given by the Fréchet mean~\cite{Frechet1948}, also known as the Karcher mean on manifolds~\cite{karcher1977riemannian}. Defined as a minimizer of expected squared distance, it provides a notion of central tendency in metric spaces and has been widely used in shape analysis and geometric statistics~\cite{pennec2006intrinsic, bhattacharya2003large, afsari2011riemannian}. Viewing periodic trajectories as closed curves in state space, defined up to circular reparametrization, places the problem naturally in the setting of statistics on nonlinear and quotient spaces.

In this work, we investigate Fréchet means for periodic trajectories of dynamical systems. A central challenge lies in the interplay between geometry and dynamics. Trajectories should be compared independently of parametrization to capture their geometry, yet parametrization itself contains dynamical information such as speed variations and fast--slow structure. To address this challenge, we develop a framework inspired by shape analysis and introduce a metric on the space of closed curves that accounts for circular phase shifts. This induces a quotient-space structure in which empirical Fréchet means can be defined and computed, and for which we establish existence of empirical Fréchet means. We show that time parametrization preserves dynamical information but may distort geometric features, whereas arc length parametrization isolates geometry at the expense of dynamics. To reconcile these viewpoints, we propose a decoupled approach. First, we compute a geometric Fréchet mean using arc length parametrization. Second, we reconstruct dynamical information along this curve via harmonic averaging of speed profiles. The resulting trajectory combines geometric consistency with a meaningful dynamical interpretation.

A related question, particularly relevant in UQ settings, is whether the resulting mean is representative of the dataset. In heterogeneous ensembles, the Fréchet mean may fail to resemble any individual sample and may introduce artificial features. To quantify this effect, we introduce diagnostic measures based on curvature and medoid-type statistics that assess when averaging yields meaningful summaries and when it does not.

We illustrate the proposed framework using three dynamical systems with increasing geometric complexity. The Van der Pol oscillator highlights the influence of parametrization on the Fréchet mean. The Rosenzweig--MacArthur predator--prey model produces families of periodic orbits with varying amplitudes and scaling behavior. Finally, the Morris--Lecar model provides an example in which the Fréchet mean may fail to represent the dataset, motivating the use of the proposed diagnostic tools.


The remainder of this paper is organized as follows. In Section~\ref{sec:frechet mean in metric spaces} we review the definition of the Fréchet mean in general metric spaces and discuss existence and uniqueness results. Section~\ref{sec:a metric structure for closed curves} introduces the metric structure on the space of closed curves that serves as the geometric framework for representing periodic trajectories, defining their Fréchet mean, and establishing its existence. In Section~\ref{sec:computation of the empirical frechet mean}, we present the algorithm for computing empirical Fréchet means on the proposed metric space and establish its convergence properties. Section~\ref{sec:reconstructing dynamics} describes how dynamical information can be reconstructed from geometrically averaged curves. Section~\ref{sec:diagnostics} introduces diagnostic quantities for assessing the representativeness of the resulting geometric Fréchet mean. Sections~\ref{sec:vdp}--\ref{sec:morris lecar model} demonstrate the framework on three example systems. Finally, Section~\ref{sec:discussion} discusses implications, limitations, and future directions. The implementation used for the numerical experiments is publicly available.\footnote{
Code and GIF animations are available at
\url{https://github.com/paulinash/FrechetMeanPaper}.
}

\section{Fr\'echet Mean in Metric Spaces}\label{sec:frechet mean in metric spaces}

To formalize the notion of an 'average' trajectory in a non-Euclidean setting, we consider the concept of the Fréchet mean in general metric spaces. This notion provides a natural extension of the classical Euclidean mean to situations where no linear structure is available. Fréchet means are widely used in areas such as shape analysis, statistics on manifolds, and geometric data analysis, where data points are given as objects like curves, surfaces, or probability distributions rather than vectors~\cite{srivastava2016functional, pennec2006intrinsic, molchanov2005theory}. This makes it a natural tool for studying spaces of trajectories, which we later model as closed curves equipped with a suitable metric. However, the Fréchet mean is not (yet) a common tool in dynamical systems, particularly at the interaction with forward uncertainty quantification, where random differential equations are naturally generating families of periodic orbits.

Consider a probability space $(\Omega, \mathcal{F}, \mathbb{P})$ and a metric space $(\mathcal{C}, d)$. Let $X: \Omega \to \mathcal{C}$ be a random element such that $\mathbb{E}[d(c_0, X)^2] < \infty$ for some $c_0 \in \mathcal{C}$. The associated Fréchet functional is defined by
\[
F(c) := \mathbb{E}[d(c, X)^2], \qquad c \in \mathcal{C}.
\]
The Fréchet mean of $X$ is defined as the (possibly non-unique) set of minimizers of $F$, i.e.,
\(
\arg\min_{c \in \mathcal{C}} F(c).
\)
Any element $\mu_F$ of this set is called a Fréchet mean of $X$. In general, the Fr\'echet mean need not be unique, and the minimizing set may be empty if the infimum is not attained. The Fréchet variance is defined as the minimal value of the functional,
\[
\mathrm{Var}_F := \inf_{c \in \mathcal{C}} F(c).
\]
If a Fréchet mean $\mu_F$ exists, then $\mathrm{Var}_F = F(\mu_F) = \mathbb{E}[d(\mu_F, X)^2]$. The concept was first introduced in~\cite{Frechet1948} and we mainly follow~\cite{molchanov2005theory} for definitions and theoretical background. In practice, one typically works with finite samples rather than distributions. This leads to the notion of an empirical Fréchet mean.
Let $\gamma_1, \ldots, \gamma_N \in \mathcal{C}$ be sample points. The empirical Fr\'echet mean is defined as any minimizer
\[
\mu_F^{(N)} \in \arg\min_{c \in \mathcal{C}} \frac{1}{N} \sum_{k=1}^{N} d(c, \gamma_k)^2.
\]
The corresponding empirical Fr\'echet variance is given by
\(
\mathrm{Var}_F^{(N)} = \inf_{c \in \mathcal{C}} \frac{1}{N} \sum_{k=1}^{N} d(c, \gamma_k)^2,
\)
and, if a minimizer $\mu_F^{(N)}$ exists, it satisfies
\(
\mathrm{Var}_F^{(N)} = \frac{1}{N} \sum_{k=1}^{N} d(\mu_F^{(N)}, \gamma_k)^2.
\)

\subsection{Existence and Uniqueness of Fréchet Means}

The existence and uniqueness of Fréchet means depend strongly on the geometry
of the underlying metric space and we briefly summarize the key results for general metric spaces and Hilbert spaces.

\paragraph{Existence and uniqueness in general metric spaces}
Let $(\mathcal{C}, d)$ be a metric space and let $X$ be a random variable
with $\mathbb{E}[d(c_0,X)^2] < \infty$ for some $c_0 \in \mathcal{C}$. Then the Fréchet
functional
\(
F(c) = \mathbb{E}[d(c,X)^2]
\)
is well-defined. However, a minimizer need not exist in general, as the infimum may not be attained~\cite{molchanov2005theory}. A sufficient condition for existence is compactness of the space $\mathcal{C}$, which ensures that the Fréchet functional attains its minimum. Even when a Fréchet mean exists, uniqueness is not guaranteed in general metric spaces. Additional geometric conditions are required, such as nonpositive curvature, which induce convexity of the squared distance functional and thereby ensures uniqueness~\cite{pennec2006intrinsic}.

\paragraph{Existence and uniqueness in Hilbert spaces}
Let $\mathcal{H}$ be a Hilbert space and $X$ a random variable with $\mathbb{E}\|X\|^2 < \infty$. The Fréchet functional
\(
F(c) = \mathbb{E}[\|c - X\|^2]
\)
admits the decomposition
\begin{align*}
   F(c) &= \mathbb{E}[\|c - X\|^2] = \|c\|^2 - 2 \langle c,\mathbb{E}[X] \rangle + \mathbb{E}\|X\|^2\\
   &= \| c - \mathbb{E}[X] \|^2 - \| \mathbb{E}[X]\|^2 + \mathbb{E}\|X\|^2\\
   &=\|c - \mathbb{E}[X]\|^2 + \mathbb{E}[\|X - \mathbb{E}[X]\|^2].
\end{align*}
Since the second term is independent of $c$, the functional is minimized at $c = \mathbb{E}[X]$. In particular, $F$ is strictly convex, and therefore admits a unique minimizer. This characterization of the expectation as the minimizer of the mean squared distance is classical in Hilbert spaces~\cite{molchanov2005theory}. For a finite sample $\gamma_1,\dots,\gamma_N \in \mathcal{H}$, the empirical Fréchet functional satisfies an analogous identity,
\begin{align*}
   F^{(N)}(c) &= \frac{1}{N}\sum_{k=1}^N \|c - \gamma_k\|^2 \\
    &= \|c - \frac{1}{N}\sum_{k=1}^N \gamma_k\|^2 +  \frac{1}{N}\sum_{k=1}^N \| \gamma_k - \frac{1}{N}\sum_{k=1}^N \gamma_k \|^2 \\
    &=  \|c - \frac{1}{N}\sum_{k=1}^N \gamma_k\|^2 +  \mathrm{Var}_F^{(N)}
\end{align*}
so the unique minimizer is given by the arithmetic mean
\(
\mu_F^{(N)} = \frac{1}{N}\sum_{k=1}^N \gamma_k.
\)Thus, in Hilbert spaces, the Fréchet mean coincides with the classical Euclidean average, reflecting the underlying linear and inner product structure. 

In contrast, the spaces of trajectories considered in this work do not directly admit such a structure, as curves are defined only up to transformations such as phase shifts. Similar challenges arise in shape analysis, where one studies curves modulo reparametrization or other group actions. Motivated by these ideas, we model periodic trajectories as elements of a quotient space that accounts for phase invariance while preserving physically meaningful properties. This quotient space inherits a metric structure from an underlying Hilbert space but is itself no longer linear. As a consequence, the existence, uniqueness, and interpretation of Fréchet means become significantly more subtle and depend sensitively on the chosen metric. 


\section{A Metric Structure for Closed Curves}\label{sec:a metric structure for closed curves}

In order to define Fréchet means for periodic trajectories, we introduce a metric structure on the space of closed curves. The construction is guided by the general framework of group actions and quotient spaces commonly used in shape analysis~\cite{srivastava2016functional}, but we deliberately restrict invariances to preserve physically meaningful properties such as location, timing, and velocity. Our approach proceeds in several steps. First, we define the space of closed curves representing periodic trajectories. Second, we introduce a circular reparametrization group acting on this space, leading to a quotient space that captures phase invariance along the orbit. Third, we equip this quotient space with a phase-aligned distance, which induces the metric structure used to define Fréchet means. We then introduce arc length parametrization as an alternative to time normalization. Finally, we discuss existence and uniqueness properties of Fréchet means in the resulting quotient space.

\paragraph{Space of Closed Curves}
We model closed curves as elements of the Hilbert space
\(
\mathcal{H} := L^2(\mathcal{S}^1, \mathbb{R}^d),
\)
where $\mathcal{S}^1 \cong [0,1]/(0\sim 1) $ denotes the unit circle. Elements of $\mathcal{H}$ are equivalence classes of square-integrable functions on the circle, which we interpret as periodic curves up to sets of measure zero. The space $\mathcal{H}$ is equipped with the inner product $\langle \cdot, \cdot \rangle_{L^2}$ and induced norm $\| \cdot \|_{L^2}$, i.e.
\[\langle \gamma_1, \gamma_2 \rangle_{L^2} = \int_{\mathcal{S}^1} \gamma_1(s) \cdot \gamma_2(s) \, \mathrm{d}s \quad \mathrm{and} \quad  \|\gamma\|_{L^2} = \sqrt{\langle \gamma, \gamma \rangle_{L^2}} =\left( \int_{\mathcal{S}^1} \|\gamma(s)\|^2 \, \mathrm{d}s \right)^{1/2},\] where $\| \cdot \|$ denotes the Euclidean norm on $\mathbb{R}^d.$
The associated metric is given by
\[
d_0(\gamma_1,\gamma_2) = \|\gamma_1 - \gamma_2\|_{L^2}.
\]
In applications, closed curves arise as periodic trajectories defined on time intervals of varying lengths. Let
\(\gamma : [t_1, t_1 + T_1] \to \mathbb{R}^d\)
be a continuous periodic curve with period $T_1 > 0$. We define time normalized parametrizations on $\mathcal{S}^1$ by
\begin{align}\label{eq: time normalization}
    \tilde{\gamma}(s) := \gamma(t_1 + sT_1), \quad s \in \mathcal{S}^1.
\end{align}
In the following, we identify periodic trajectories with their normalized representatives in $\mathcal{H}$.
While the metric structure is defined on $\mathcal{H} = L^2(\mathcal{S}^1,\mathbb{R}^d)$,
certain geometric constructions introduced in later sections require additional smoothness. In particular,
arc length parametrization requires continuously differentiable curves, and
curvature requires twice continuously differentiable curves. To this end, we consider the subspaces
\[
C^1(\mathcal{S}^1,\mathbb{R}^d) \subset \mathcal{H}, \qquad
C^2(\mathcal{S}^1,\mathbb{R}^d) \subset C^1(\mathcal{S}^1,\mathbb{R}^d).
\] These spaces consist of periodic functions on $\mathcal{S}^1$ whose derivatives match at
the endpoints, and can be identified with $C^k$ functions on the circle $\mathcal{S}^1$.

\paragraph{Quotient Space by Circular Reparametrization}
Closed curves do not have a canonical starting point. 
To remove this phase ambiguity, we consider circular shifts of the parameter domain. Let $G := \mathcal{S}^1$ denote the set of all shifts $\tau \in \mathcal{S}^1$. For $\tau \in G$ and $\gamma \in \mathcal{H}$, define the shifted curve
\[(\tau \cdot \gamma)(s) := \gamma(s + \tau),\] where addition is understood modulo $1$ on $\mathcal{S}^1$. This operation simply changes the starting point of the curve without altering its geometric shape.

\begin{lemma}
The mapping $(\tau, \gamma) \mapsto \tau \cdot \gamma$ defines a continuous group action on $\mathcal{H}=L^2(\mathcal{S}^1,\mathbb{R}^d)$ that preserves the $L^2$-distance $d_0$.
\end{lemma}

\begin{proof}
The identity element $0 \in G$ satisfies $0 \cdot \gamma = \gamma$, and for $\tau_1, \tau_2 \in G$,
\(
\tau_1 \cdot (\tau_2 \cdot \gamma) = (\tau_1 + \tau_2) \cdot \gamma,
\)
so this defines a group action. For $\tau \in G$, a change of variables yields
\[
d_0(\tau \cdot \gamma_1, \tau \cdot \gamma_2)^2
= \int_0^1 \|\gamma_1(s+\tau) - \gamma_2(s+\tau)\|^2 \mathrm{d}s
= d_0(\gamma_1, \gamma_2)^2,
\]
so the action preserves distances. Continuity follows since shifts are continuous in $L^2$.
\end{proof}
We consider curves equivalent if they differ only by such a shift, i.e.
\[
\gamma_1 \sim \gamma_2 \;\Longleftrightarrow\; \gamma_1 = \tau \cdot \gamma_2 \text{ for some } \tau \in G.
\]
The corresponding quotient space is
\(\mathcal{C} := \mathcal{H} / \mathcal{S}^1 = L^2(\mathcal{S}^1,\mathbb{R}^d)/\mathcal{S}^1,\)
whose elements represent closed curves modulo their starting point. Throughout, we denote elements of $\mathcal{H}$ by $\gamma$ or $c$, and write $[\gamma]$ or $[c]$ for their equivalence classes in $\mathcal{C} = \mathcal{H} / \mathcal{S}^1$. We consistently use the bracket notation when emphasizing membership in the quotient space.

\paragraph{Phase-Aligned Distance}
To compare curves up to circular reparametrization, we define
\[
d([\gamma_1], [\gamma_2]) := \min_{\tau \in \mathcal{S}^1} d_0(\gamma_1, \tau \cdot \gamma_2),
\] where the right-hand side is defined via representatives. The proposition below shows it is well-defined on equivalence classes and yields a metric on $\mathcal{C}.$

\begin{proposition}
The minimum is attained and $d$ defines a metric on the quotient space $\mathcal{C}$.
\end{proposition}

\begin{proof}
For fixed $\gamma_1, \gamma_2 \in \mathcal{H}$, the function
\(
\tau \mapsto d_0(\gamma_1, \tau \cdot \gamma_2)
\)
is continuous, and $\mathcal{S}^1$ is compact, hence the minimum exists.
The invariance of $d_0$ under the group action implies that $d$ depends only on equivalence classes.
Non-negativity, symmetry, and the triangle inequality follow from the corresponding properties of $d_0$.
If $d([\gamma_1], [\gamma_2]) = 0$, then there exists a $\tau^*$ such that $d_0(\gamma_1,\tau^* \cdot \gamma_2) = 0$, as the minimum is attained. Thus  $\gamma_1(s) = (\tau^* \cdot \gamma_2)(s)$ a.e. and the equivalence classes coincide.
\end{proof}
Thus, we obtain a metric space $(\mathcal{C}, d)$, where $\mathcal{C}$ is the quotient of the Hilbert space $\mathcal{H} = L^2(\mathcal{S}^1,\mathbb{R}^d)$ under circular shifts, enabling comparison of periodic trajectories independently of their starting point.
A visualization of the circular shifts operation and resulting effects on the naive pointwise mean are shown in Figure~\ref{fig:phase_alignment}.

\begin{figure}[h]
\centering

\begin{subfigure}[b]{0.37\linewidth}
    \centering
    \includegraphics[width=\linewidth]{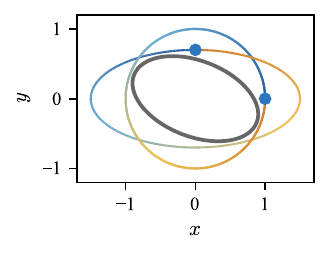}
    \caption{Unaligned phase}
\end{subfigure}
\begin{subfigure}[b]{0.37\linewidth}
    \centering
    \includegraphics[width=\linewidth]{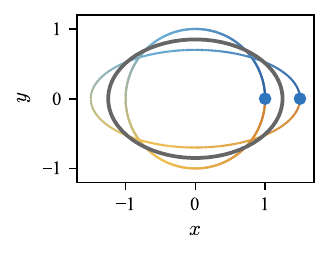}
    \caption{Aligned phase}
\end{subfigure}

\caption{Effect of phase alignment on pointwise averaging of closed curves. (a) Two periodic curves with unaligned parametrizations produce a distorted and geometrically unnatural pointwise average  (grey) in $L^2$. (b) After optimal alignment by circular phase shifts using the metric $d$, the resulting average curve preserves the underlying geometry much more faithfully. This illustrates the importance of phase alignment in the definition of the quotient-space metric on $\mathcal{C}$.}
\label{fig:phase_alignment}
\end{figure}

\paragraph{Arc Length Parametrization}
In addition to the time normalization introduced in \eqref{eq: time normalization}, we consider an arc length parametrization. While time parametrization preserves local speeds along the curve and thus reflects the underlying dynamics, arc length parametrization enforces constant speed and removes variability due to parametrization, isolating the intrinsic geometric structure of periodic trajectories~\cite{srivastava2016functional}.
Arc length parametrization also appears in dynamical systems, for instance in pseudo-arclength continuation methods, where solution branches are parametrized by arc length to allow continuation through folds and turning points~\cite{allgower2012numerical, kuznetsov1998elements}. This highlights its role as a geometrically intrinsic parametrization that avoids degeneracies of non-uniform parametrizations. In our setting, arc length parametrization provides a natural way to define a geometric Fréchet mean of periodic trajectories, which can later be complemented by a separate reconstruction of dynamical information in Section~\ref{sec:reconstructing dynamics}. In particular, it defines an alternative to the proposed time normalization, allowing us to compare how different parametrizations influence the resulting Fréchet mean. 
Let $\gamma : [t_1, t_1 + T_1] \to \mathbb{R}^d$ be a continuously differentiable regular closed curve with $\|\dot{\gamma}(t)\| \neq 0$ and define the cumulative arc length  $l$ and the resulting total arc length $L$
\[l(t) = \int_{t_1}^{t} \|\dot{\gamma}(\tau)\| d\tau, \quad L = l(t_1 + T_1),\]
and the normalized arc length parameter
$s(t) = l(t)/L\in \mathcal{S}^1.$
Since $l$ is strictly increasing, it admits an inverse, and we define
\[\overline{\gamma}(s) := \gamma\bigl(l^{-1}(sL)\bigr), \quad s \in \mathcal{S}^1.\]
Thus, $\overline{\gamma}$ yields an arc length parametrization of the initial curve $\gamma$, tracing the same geometric curve with constant speed. In the following, whenever arc length parametrization is used, we implicitly assume that curves lie in $C^1(\mathcal{S}^1,\mathbb{R}^d)$.

\subsection{Existence and Uniqueness in the Quotient Space}
We now analyze existence and uniqueness of the empirical Fréchet mean 
in the quotient space $\mathcal{C} = L^2(\mathcal{S}^1,\mathbb{R}^d)/\mathcal{S}^1$. 
The space $\mathcal{C}$ is not locally compact, since bounded sets in the 
infinite-dimensional space $\mathcal{H}$ are not norm-compact. Standard 
compactness arguments for the existence of minimizers of the empirical 
Fréchet functional
\[
F^{(N)}([c]) = \frac{1}{N} \sum_{k=1}^N d([c],[\gamma_k])^2
\]
therefore do not apply directly. To overcome this, we lift the minimization 
problem to the ambient Hilbert space $\mathcal{H}$, replacing 
$\min_{[c] \in \mathcal{C}} F^{(N)}([c])$ by
\[
\min_{c \in \mathcal{H}} F^{(N)}(c), \qquad F^{(N)}(c) = \frac{1}{N} 
\sum_{k=1}^N \min_{\tau \in \mathcal{S}^1} \|c - \tau \cdot \gamma_k\|_{L^2}^2.
\]
The two problems are equivalent: since $F^{(N)}(c) = F^{(N)}(\tau \cdot c)$ 
for all $\tau \in \mathcal{S}^1$, the lifted functional depends only on the 
equivalence class $[c] \in \mathcal{C}$. However, the inner minimization 
over $\tau$ destroys the quadratic structure of $F^{(N)}$, so standard 
convexity arguments in Hilbert spaces do not apply either. We therefore 
proceed via the direct method in the calculus of 
variations~\cite{dacorogna2008direct, brezis2011functional}.

\begin{proposition}\label{prop:existence minimizer}
The empirical Fréchet functional $F^{(N)}\colon \mathcal{H} \to \mathbb{R}$,
\[
    F^{(N)}(c) = \frac{1}{N}\sum_{k=1}^N 
    \min_{\tau \in \mathcal{S}^1}\|c - \tau\cdot\gamma_k\|_{L^2}^2,
\]
admits a minimizer $c^* \in \mathcal{H}$. Its equivalence class 
$[c^*] \in \mathcal{C}$ is an empirical Fréchet mean.
\end{proposition}

\begin{proof}
    We apply the direct method in the calculus of variations, proceeding in three steps.

    \medskip
    \textit{Step 1: Minimizing sequence and weak convergence:}
    Since $F^{(N)}\geq 0$, the functional is bounded below and admits a minimizing sequence $(c_n) \subset \mathcal{H}$ with $F^{(N)} (c_n) \leq M < \infty$ for all $n$. For fixed $k=1$, the map $\tau \mapsto \|c_n - \tau \cdot \gamma_1\|_{L^2}^2$ is continuous on the compact set $\mathcal{S}^1$, so its minimum is attained at some $\tau_n \in \mathcal{S}^1$. The triangle inequality and $\mathcal{S}^1$-invariance of the $L^2$-norm then give 
    \[ \|c_n\|_{L^2} \leq \|c_n - \tau_n \cdot \gamma_1\|_{L^2} + \|\tau_n \cdot \gamma_1\|_{L^2} \leq \sqrt{NM} + \|\gamma_1\|_{L^2}.\]
    Hence $(c_n)$ is bounded in $\mathcal{H}$, and by reflexivity a subsequence satisfies $c_{n_j} \rightharpoonup c^*$ weakly in $\mathcal{H}$.

    \medskip
    \textit{Step 2: Weak lower semicontinuity:} For each $k=1,\dots,N$, define
    \[\phi_k(c) = \min_{\tau \in \mathcal{S}^1} \|c - \tau \cdot \gamma_k\|_{L^2}^2\ = \mathrm{dist}(c,\mathcal{O}_k)^2,\] where $\mathcal{O}_k=\{\tau \cdot \gamma_k : \tau \in \mathcal{S}^1\}$ is the orbit of $\gamma_k$ under the group action. We show that $\phi_k$ is weakly lower semicontinuous on $\mathcal{H}$. Let $c_n \rightharpoonup c^*$ weakly in $\mathcal{H}$. Choose a subsequence $(c_{n_j})$ such that $\phi_k(c_{n_j}) \to \liminf_n \phi_k(c_n)$. In particular, $c_{n_j}$ also converges weakly. For each $j$, let $\tau_j \in \mathcal{S}^1$  be a minimizer for $\| c_{n_j} - \tau \cdot \gamma_k\|_{L^2}^2$, which exists by compactness of $\mathcal{S}^1$ and continuity. Defining $h_j = \tau_j \cdot \gamma_k\in \mathcal{O}_k$ then yields
    \[\|c_{n_j} - h_j\|_{L^2}^2 = \phi_k(c_{n_j}).\]The orbit $\mathcal{O}_k$ is compact in $\mathcal{H}$ since it is the continuous image of the compact set $\mathcal{S}^1$. This admits a convergent subsequence (reindexed) $h_j \to h^*\in \mathcal{O}_k$ strongly in $\mathcal{H}$. Since $c_{n_{j}} \rightharpoonup c^*$ weakly and $h_{j} \to h^*$ strongly, their difference satisfies
    \[c_{n_{j}} - h_{j}  \rightharpoonup c^* - h^* \quad \text{weakly in } \mathcal{H}.\]
    The squared $L^2$-norm is convex and strongly continuous and thus weakly lower semicontinuous. This yields
    \[\phi_k(c^*) \leq \|c^* - h^* \|_{L^2}^2 \leq \liminf_j \|c_{n_{j}} - h_{j}\|_{L^2}^2 = \liminf_j \phi_k(c_{n_{j}}) = \liminf_n \phi_k(c_n),\]
     i.e. $\phi_k$ is weakly lower semicontinuous. As a weighted finite sum, the Fréchet functional $F^{(N)}(c)$ is thus also weakly lower semicontinuous, i.e. 
    \(F^{(N)} (c^*) \leq \liminf_n F^{(N)} (c_n).\)

    \medskip
    \textit{Step 3: Existence of minimizer:} Applying weak lower semicontinuity of $F^{(N)}$ to the minimizing sequence from Step 1, we obtain
    \[F^{(N)}(c^*) \;\leq\; \liminf_{j\to\infty} F^{(N)}(c_{n_j}) \;=\; \inf_{c\in\mathcal{H}} F^{(N)}(c).\]
    Since $F^{(N)}(c^*) \geq \inf_{c \in \mathcal{H}} F^{(N)}(c)$ holds trivially, we conclude $F^{(N)}(c^*) =\inf_{c\in\mathcal{H}} F^{(N)}(c)$, so $c^*$ is a minimizer. Its equivalence class $[c^*] \in \mathcal{C}$ is an empirical Fréchet mean.
\end{proof}
The proof follows the same direct-method strategy as the metric-space existence results of \cite{bhattacharya2003large}: extract a minimizing sequence, show it is bounded, and obtain a convergent subsequence from which existence follows by (semi-)continuity of the functional. But the present setting differs in two essential ways. 
First, since $\mathcal{H}$ is infinite-dimensional, norm-compactness of bounded sets fails; compactness of the minimizing sequence is therefore replaced by weak sequential compactness via reflexivity of $\mathcal{H}$. Second, continuity of the Fréchet functional is replaced by weak lower semicontinuity, which holds here because each summand is a squared distance to a compact orbit. Existence and uniqueness of Fréchet means have also been studied on Riemannian manifolds using geodesic convexity and curvature bounds \cite{karcher1977riemannian, kendall1990probability, pennec2006intrinsic}, and under support concentration conditions on shape spaces \cite{le2001locating, ziezold1977expected}; these results do not apply here since $\mathcal{C}$ carries no Riemannian structure.

In contrast to the Hilbert space setting, uniqueness of the empirical Fréchet mean generally fails in $\mathcal{C}$. The functional $F^{(N)}([c])$ is not convex on $\mathcal{C}$ due to the minimization over phase shifts in the definition of the distance. Equivalently, the lifted functional $ F^{(N)}(c)$
is not convex on $\mathcal{H}$. As a result, multiple distinct minimizers in $\mathcal{C}$ may exist, which also affects the computational properties of the algorithm discussed in Section~\ref{sec:computation of the empirical frechet mean}.

\section{Computation of the Empirical Fr\'echet Mean}
\label{sec:computation of the empirical frechet mean}

In the following, we consider an iterative algorithm for approximating the empirical Fr\'echet mean of a collection of closed curves. Let $\{\gamma_1,\dots,\gamma_N\}\subset\mathcal H$ be a sample of closed curves and denote their equivalence classes by $[\gamma_k]\in\mathcal C$. The empirical Fr\'echet functional on the quotient space $\mathcal C$ is given by
\[
F^{(N)}([c])
=
\frac{1}{N}\sum_{k=1}^N d([c],[\gamma_k])^2.
\]
Working with representatives $c\in\mathcal H$, this can be written as
\[
F^{(N)}(c)
=
\frac{1}{N}
\sum_{k=1}^N
\min_{\tau\in\mathcal S^1}
\|c-\tau\cdot\gamma_k\|_{L^2}^2.
\]
For a given curve $c\in\mathcal H$, we denote by
\[
\tau_k(c)
\in
\arg\min_{\tau\in\mathcal S^1}
\|c-\tau\cdot\gamma_k\|_{L^2}^2
\]
an optimal circular shift of $\gamma_k$ with respect to $c$.
The sample curves may be normalized using either time or arc length parametrization, as discussed in Section~\ref{sec:a metric structure for closed curves}. The proposed algorithm alternates between an alignment step, in which optimal phase shifts are computed, and an averaging step, in which the aligned curves are averaged. Such block coordinate descent schemes are used in shape analysis and functional data analysis to compute Karcher means under group actions~\cite[Algorithm 31]{srivastava2016functional}. The present algorithm adapts this framework to circular reparametrizations, see Algorithm~\ref{alg:frechetmean}.
\begin{algorithm}[h]
\caption{Alternating alignment and averaging for the empirical Fréchet mean}
\label{alg:frechetmean}
\begin{algorithmic}[1]
\REQUIRE Curves $\gamma_1,\dots,\gamma_N\in\mathcal H$, tolerance $\varepsilon>0$, maximum number of iterations $I_{\max}$

\STATE Initialize
$\mu_0=\frac1N\sum_{k=1}^N\gamma_k$

\STATE Compute optimal shifts
$\tau_{k,0}=\tau_k(\mu_0)$,
$k=1,\dots,N$

\STATE Compute $F^{(N)}(\mu_0)$ using the optimal shifts $\tau_{k,0}$

\FOR{$n=0,\dots,I_{\max}$}

\STATE Update the mean
$\mu_{n+1}
=\frac1N\sum_{k=1}^N(\tau_{k,n}\cdot\gamma_k)$

\STATE Re-align the sample curves by computing
$\tau_{k,n+1}=\tau_k(\mu_{n+1})$,
$k=1,\dots,N$

\STATE Compute $F^{(N)}(\mu_{n+1})$ using the optimal shifts $\tau_{k,n+1}$

\IF{$
\dfrac{|F^{(N)}(\mu_{n+1})-F^{(N)}(\mu_n)|}
{F^{(N)}(\mu_n)+\delta}
<\varepsilon
$}
    \STATE \textbf{break}
\ENDIF

\ENDFOR

\RETURN $\mu_{n+1}$ as an approximation of $\mu_F^{(N)}$
\end{algorithmic}
\end{algorithm}

\subsection{Convergence of the Algorithm}

We analyze the convergence properties of the iterative scheme. 
The method constitutes a block coordinate descent algorithm applied to the functional
\[
\mathcal{G}(c,\boldsymbol{\tau}=(\tau_1,\dots,\tau_N))
=
\frac{1}{N}\sum_{k=1}^N
\int_0^1
\|c(s)-(\tau_k \cdot \gamma_k)(s)\|^2 ds,
\]
defined on the product space \(\mathcal{H} \times (\mathcal{S}^1)^N\), yielding the relation
\(F^{(N)}(c) = \min_{\boldsymbol{\tau} \in (\mathcal{S}^1)^N} \mathcal{G}(c,\boldsymbol{\tau}).\)
\begin{theorem}\label{theo:convergence}
Let $(\mu_n)$ be the sequence generated by the algorithm. Then:

\medskip
\noindent\textnormal{(i) Monotonicity:}
\(
F^{(N)}(\mu_{n+1}) \le F^{(N)}(\mu_n)\) for all $n$.

\noindent\textnormal{(ii) Convergence of functional values:}
The sequence $\bigl(F^{(N)}(\mu_n)\bigr)$ converges to some
$F_\infty \ge 0$.

\noindent\textnormal{(iii) Boundedness:}
The sequence $(\mu_n)$ is bounded in
$L^2(\mathcal S^1,\mathbb R^d)$.

\noindent\textnormal{(iv) Subsequential convergence:}
There exists a subsequence $(\mu_{n_j})$ and a limit
$\mu_\ast\in\mathcal H$ such that
\(
\mu_{n_j}\rightharpoonup\mu_\ast\)
weakly in $L^2$.
\end{theorem}

\begin{proof}
\textit{Step 1: Monotonicity:}
The algorithm alternates exact minimization steps of the functional $\mathcal{G}$.
For fixed $\mu_n$, the alignment step yields an optimal $\boldsymbol{\tau}_n = (\tau_{1,n},\dots,\tau_{N,n})$ such that
\(
\mathcal{G}(\mu_n,\boldsymbol{\tau}_{n})
= \min_{\boldsymbol{\tau}} \mathcal{G}(\mu_n,\boldsymbol{\tau})
= F^{(N)}(\mu_n).
\)
For fixed $\boldsymbol{\tau}_{n}$, the update step in the following iteration then minimizes the strictly convex 
quadratic functional 
\(c \mapsto\mathcal{G}(c,\boldsymbol{\tau_n})\) in $c$, hence resulting in a unique minimizer $\mu_{n+1}$ for which it holds
\(
\mathcal{G}(\mu_{n+1},\boldsymbol{\tau}_{n})
\le \mathcal{G}(\mu_n,\boldsymbol{\tau}_{n}).
\)
Since $F^{(N)}(\mu_{n+1}) = \min_{\boldsymbol{\tau}}\mathcal{G}(\mu_{n+1},\boldsymbol{\tau})
\le \mathcal{G}(\mu_{n+1},\boldsymbol{\tau}_{n})$, combining the two inequalities gives
\(
F^{(N)}(\mu_{n+1}) \le F^{(N)}(\mu_n).
\)

\medskip
\textit{Step 2: Convergence of functional values:}
Since $F^{(N)} \ge 0$, the monotone decreasing sequence 
$\bigl(F^{(N)}(\mu_n)\bigr)$ is bounded below and therefore converges to 
some $F_\infty \ge 0$.

\medskip
\textit{Step 3: Boundedness:}
By the update formula and the triangle inequality,
\[
\|\mu_{n+1}\|_{L^2}
= \left\|\frac{1}{N}\sum_{k=1}^N (\tau_{k,n}\cdot\gamma_k)\right\|_{L^2}
\le \frac{1}{N}\sum_{k=1}^N \|(\tau_{k,n}\cdot\gamma_k)\|_{L^2}
= \frac{1}{N}\sum_{k=1}^N \|\gamma_k\|_{L^2},
\]
where the last equality holds because circular shifts are $L^2$-isometries.
The right-hand side is finite and independent of $n$, so $(\mu_n)$ is 
uniformly bounded in $L^2(\mathcal{S}^1,\mathbb{R}^d)$.

\medskip
\textit{Step 4: Subsequential convergence:}
Since $L^2(\mathcal{S}^1,\mathbb{R}^d)$ is a reflexive Hilbert space, the bounded sequence $(\mu_{n})$ from step 3 admits a weakly convergent subsequence $(\mu_{n_j})$ with
$\mu_{n_j} \rightharpoonup \mu_\ast \in L^2(\mathcal{S}^1,\mathbb{R}^d)$. 
\end{proof}

The previous theorem establishes convergence of the functional values and weak convergence of a subsequence of the iterates in $L^2$. However, these results alone do not provide a complete characterization of the asymptotic behavior of the algorithm. Even though it is shown that a subsequence $(\mu_{n_j})$ exists, which converges weakly in $\mathcal{H}$ to some limit $\mu^\ast$, the theorem does not characterize the nature of such limit points. 
In particular, without additional assumptions, one cannot guarantee that 
$\mu^\ast$ is a (local or global) minimizer of $F^{(N)}$. This limitation 
stems from the fact that the functional
$\mathcal{G}(c,\boldsymbol{\tau})$ is not jointly convex in $(c,\boldsymbol{\tau})$ 
due to the minimization over phase shifts. As a result, the functional may admit multiple stationary points or local minima, and the algorithm can therefore converge to different limit points depending on the initialization. We also emphasize that the optimization is performed in the Hilbert space 
$\mathcal{H} = L^2(\mathcal{S}^1,\mathbb{R}^d)$, and therefore any weak limit  $\mu^\ast$ is, in general, only guaranteed to lie in 
$\mathcal{H}$. In particular, the $L^2$-metric does not control derivatives, 
and no regularity or smoothness of limit points can be ensured. Nevertheless, if the input curves $\{\gamma_k\}$ belong to 
$C^1(\mathcal{S}^1,\mathbb{R}^d)$ or $C^2(\mathcal{S}^1,\mathbb{R}^d)$, then each finite 
iterate $\mu_n$ is given by an average of shifted curves and therefore 
inherits the same regularity. Hence, all iterates remain in $C^1(\mathcal{S}^1,\mathbb{R}^d)$ or $C^2(\mathcal{S}^1,\mathbb{R}^d)$, 
respectively, allowing us to consider the derivative and curvature of the iterates in the following sections. Finally, we note that the empirical Fréchet means are consistent, i.e. as $N \to \infty$, the set of empirical minimizers converges into the stochastic Fréchet mean set almost surely, provided the latter is non-empty and $\mathbb{E}[d(c_0,X)^2] < \infty$ for some $c_0 \in \mathcal{C}$~\cite{molchanov2005theory}.

The iterative scheme developed in this section forms the computational basis to study 
the van der Pol system in Section~\ref{sec:vdp}, the Rosenzweig--MacArthur model in Section~\ref{sec:rosenzweig-macarthur model}, and the Morris--Lecar system in Section~\ref{sec:morris lecar model}.

\section{Averaging of Dynamics for Geometric Fréchet Means}\label{sec:reconstructing dynamics}

Up to this point, we have considered two normalisations of periodic trajectories:
canonical time normalisation and arc length parametrisation, see Section~\ref{sec:a metric structure for closed curves}.  
While time normalisation preserves temporal dynamics, the resulting Fréchet mean may distort the geometric characteristics of the sample curves. Arc length parametrisation, in contrast, yields a geometrically meaningful Fréchet mean curve, but removes the underlying dynamical information.  We therefore introduce a procedure to reconstruct dynamics for Fréchet means computed from arc length parametrised curves.

The main idea is to treat geometry and dynamics as complementary
components.  The geometric component is obtained by averaging curves in
arc length parametrisation, yielding an approximation to the
geometric Fréchet mean $\mu_F$.  The dynamical component is represented
by scalar speed profiles along the common arc length coordinate, which
are averaged separately and then recombined with the geometric mean
curve to produce a representative mean trajectory.  


\paragraph{Representation of Dynamical Information}
Let $\gamma_k\colon[t_k, t_k+T_k]\to\mathbb{R}^d$ be a continuously
differentiable, regular, closed curve with period $T_k$, cumulative arc length $l_k$ and total arc
length $L_k$.  Its physical speed is $v_k(t)=\|\dot\gamma_k(t)\|$. Denoting the normalised arc length parameter by $s\in \mathcal{S}^1$, the arc length parametrization of the curve $\gamma_k$ and the corresponding pullback of the speed to arc length coordinates are given by
\begin{align*}
    \overline{\gamma}_k(s)
  &= \gamma_k\bigl(l_k^{-1}(sL_k)\bigr), \qquad s\in \mathcal{S}^1\\
  \overline v_k(s)
  &= v_k\!\left(l_k^{-1}(s L_k)\right), \qquad s\in \mathcal{S}^1.
\end{align*}
During the iterative computation of the geometric Fr\'echet mean, each
curve $\overline \gamma_k$ is aligned to the current mean $\mu_i$ via an optimal
cyclic phase shift $\tau_{k,i}\in \mathcal{S}^1$.  The final
optimal shift $\tau_{k,n}$ belonging to $\mu_n$ is applied to the speed
profile
\(
  \overline v_k^{\,\tau}(s)
  = (\tau_{k,n} \cdot \overline v_k) (s),
\) yielding a shifted speed profile consistent with the aligned curve $(\tau_{k,n} \cdot \overline{\gamma}_k)(s)$.

\paragraph{Harmonic Average of Dynamic Information}
The representative dynamics along the geometric Fr\'echet mean are
constructed by averaging local traversal times rather than speeds
directly. For a trajectory with speed profile
$\overline{v}_k^{\,\tau}(s)$, traversing an infinitesimal arc length
element $\mathrm{d}\ell$ requires the local traversal time
\(
  \mathrm{d}t_k
  =
  \mathrm{d}\ell/\overline{v}_k^{\,\tau}(s).
\)
This follows directly from the definition of speed
$v = \mathrm{d}\ell/\mathrm{d}t$
for regular curves \cite{pressley2010elementary}. The reciprocal velocity $1/v$ may therefore be interpreted as a local traversal-time density. Since traversal times add linearly along trajectories, it is natural to
define the average local traversal time by the arithmetic mean
\begin{align}\label{eq:dt_mean}
  \mathrm{d}t_{\mathrm{mean}}
  =
  \frac{1}{N}\sum_{k=1}^N \mathrm{d}t_k
  =
  \frac{1}{N}\sum_{k=1}^N
  \frac{\mathrm{d}\ell}{\overline{v}_k^{\,\tau}(s)}
  =
  \mathrm{d}\ell
  \left(
    \frac{1}{N}\sum_{k=1}^N
    \frac{1}{\overline{v}_k^{\,\tau}(s)}
  \right).
\end{align}
We now seek a representative speed profile
$v_{\mathrm{mean}}(s)$ such that the averaged traversal time again
admits the standard kinematic form
\(
  \mathrm{d}t_{\mathrm{mean}}
  =
  \mathrm{d}\ell/v_{\mathrm{mean}}(s).
\)
Comparing with \eqref{eq:dt_mean} yields
\begin{align*}
    \frac{1}{v_{\mathrm{mean}}(s)}
    &=
    \frac{1}{N}\sum_{k=1}^N
    \frac{1}{\overline{v}_k^{\,\tau}(s)},\\
    v_{\mathrm{mean}}(s)
    &=
    \left(
      \frac{1}{N}\sum_{k=1}^N
      \frac{1}{\overline{v}_k^{\,\tau}(s)}
    \right)^{-1}.
\end{align*}
This is precisely the harmonic mean of the aligned speed profiles, and
we therefore define
\(
  v_{\mathrm{harm}}(s)
  :=
  v_{\mathrm{mean}}(s).
\)
This principle already appears in the classical average-speed problem for motion over fixed distances \cite{falk2005average}. Suppose a traveller moves the same distance $L$
outward with speed $v_1$ and returns the same distance with speed
$v_2$. The total travel time is
\[
  T
  =
  \frac{L}{v_1}
  +
  \frac{L}{v_2},
\]
while the total travelled distance is $2L$. The effective average speed
is therefore
\[
  v_{\mathrm{eff}}
  =
  \frac{2L}{
    \frac{L}{v_1}+\frac{L}{v_2}
  }
  =
  \frac{2}{
    \frac{1}{v_1}+\frac{1}{v_2}
  },
\]
which is the harmonic mean of $v_1$ and $v_2$. The reason is that travel
times, rather than velocities themselves, are additive. In contrast, the arithmetic mean of the aligned speed profiles does not correspond to averaging traversal times and therefore produces a systematic underestimation of the reconstructed traversal time.
Indeed, since $\phi(x)=1/x$ is strictly convex on
$\mathbb{R}_{>0}$, Jensen's inequality yields
\[
  \frac{1}{v_{\mathrm{arith}}(s)}
  =
  \phi\!\left(
    \frac{1}{N}\sum_k
    \overline{v}_k^{\,\tau}(s)
  \right)
  \leq
  \frac{1}{N}\sum_k
  \phi\!\left(
    \overline{v}_k^{\,\tau}(s)
  \right)
  =
  \frac{1}{N}\sum_{k=1}^N
  \frac{1}{\overline{v}_k^{\,\tau}(s)}
  =
  \frac{1}{v_{\mathrm{harm}}(s)},
\]
with equality if and only if all aligned speed profiles coincide at
$s$. Hence
\(
  v_{\mathrm{arith}}(s)
  \geq
  v_{\mathrm{harm}}(s)
\)
holds pointwise, so the arithmetic reconstruction traverses the mean
curve strictly faster and therefore yields a strictly shorter
reconstructed period whenever the speed profiles are not identical.

\paragraph{Dynamic Reconstruction of Fréchet Mean}
The iterative algorithm in Section~\ref{sec:computation of the empirical frechet mean} yields $\mu_n$ as the pointwise average of the
aligned, arc length parametrised sample curves and we assume $\mu_n$ to be regular\footnote{The regularity assumption $\| \mu_n' \| > 0$ is non-trivial in general. The pointwise average of arc length parametrised curves can develop near-zero speed where the sample curves carry opposing tangent directions. In practice, however, regularity holds whenever the sample orbits are geometrically close, which is the relevant regime for a meaningful Fréchet mean.}. Since the pointwise
average of arc length parametrised curves is, in general, not itself arc
length parametrised, the speed $\|\mu_n'(s)\|$ is not identically
constant and must be accounted for explicitly in the time map.  At arc length
position $s$, traversing the infinitesimal arc length element
$\mathrm{d}\ell = \|\mu_n'(s)\|\,\mathrm{d}s$ at speed
$v_{\mathrm{harm}}(s)$ takes local time
\(
  \mathrm{d}t=
  \|\mu_n'(s)\|\,\mathrm{d}s/v_{\mathrm{harm}}(s).
\)
Integrating from $0$ to $s$ yields the cumulative time map
\[
  \theta_\mathrm{harm}(s)
  = \int_0^s \frac{\|\mu_n'(\sigma)\|}{v_{\mathrm{harm}}(\sigma)}\,
    \mathrm{d}\sigma
  = \frac{1}{N}\sum_{k=1}^N
    \int_0^s \frac{\|\mu_n'(\sigma)\|}{\overline{v}_k^{\,\tau}(\sigma)}\,
    \mathrm{d}\sigma,
  \qquad s\in \mathcal{S}^1,
\]
where the second equality follows by linearity of integration and the
definition of $v_{\mathrm{harm}}$. Due to the regularity of $\mu_n$, the cumulative time map $\theta_\mathrm{harm}$ is invertible and we obtain the dynamically
reconstructed mean trajectory via
\[
  \mu_n^{\mathrm{dynamic}}(t)
  = \mu_n\!\left(\theta_{\mathrm{harm}}^{-1}(t)\right),
  \qquad t\in[0,T_{\mathrm{harm}}],
\]
with period $T_{\mathrm{harm}} := \theta_{\mathrm{harm}}(1)$.  We thus
obtain a mean trajectory that captures both the typical geometry and the
typical dynamical behaviour of the sample.

\paragraph{Example}
To demonstrate the separation of geometric and dynamical information
concretely, consider two trajectories sharing the same geometry, describing a circle, but
differing in their temporal parametrisation
\begin{align*}
    \gamma_1(s) &= \gamma_2(s) = \bigl(\cos(2\pi s),\,\sin(2\pi s)\bigr),
  \quad s\in \mathcal{S}^1, \\
  v_1(s) &= 1, \qquad
  v_2(s) = 1 + \varepsilon\sin(2\pi s), \quad 0 < \varepsilon < 1,
\end{align*} see Figure~\ref{fig:speed_profiles_and_time_maps}.
Since the geometric curves coincide and are already aligned, the
approximation to the geometric Fr\'echet mean is given by $\mu_n(s) = (\cos(2\pi s), \sin(2\pi s))$.
The individual cumulative time maps are
\[
  \theta_k(s) = \int_0^s \frac{L}{v_k(\sigma)}\,\mathrm{d}\sigma,
  \qquad L = 2\pi,
\]
with periods $T_k = \theta_k(1)$.  These are the physical times at
which each trajectory reaches arc length position $s$; see Figure~\ref{fig:speed_profiles_and_time_maps}. As the mean curve $\mu_n$ geometrically aligns with $\gamma_1$ and $\gamma_2$, it admits the same arc length
$L = 2\pi$, so $\|\mu_n'(\sigma)\| = L$ for all $\sigma$.  The harmonic and arithmetic time maps---resulting from the harmonic and arithmetic means $v_\mathrm{harm}$ and $v_\mathrm{arith}$---are thus given by
\[
  \theta_{\mathrm{harm}}(s)
  = \int_0^s \frac{L}{v_{\mathrm{harm}}(\sigma)}\,\mathrm{d}\sigma,
  \qquad
  \theta_{\mathrm{arith}}(s)
  = \int_0^s \frac{L}{v_{\mathrm{arith}}(\sigma)}\,\mathrm{d}\sigma.
\]
In this special case, because $\|\mu_n'(\sigma)\| = L_k = L$ for both
curves, the harmonic time map satisfies
\begin{align}\label{eq:theta_harm}
  \theta_{\mathrm{harm}}(s)
  &= \int_0^s \frac{L}{v_{\mathrm{harm}}(\sigma)}\,\mathrm{d}\sigma
   = \frac{1}{2}\sum_{k=1}^2 \int_0^s \frac{L}{v_k(\sigma)}\,\mathrm{d}\sigma
   = \tfrac{1}{2}\bigl(\theta_1(s) + \theta_2(s)\bigr),
\end{align}
so the reconstructed period satisfies $T_{\mathrm{harm}} =
\theta_{\mathrm{harm}}(1) = \tfrac{1}{2}(T_1+T_2)$, i.e.\ it
coincides with the arithmetic mean of the sample periods. Applying the pointwise inequality $v_{\mathrm{arith}}(s) \geq
v_{\mathrm{harm}}(s)$ yields
\(
  \theta_{\mathrm{arith}}(s) \leq \theta_{\mathrm{harm}}(s),
\)
with strict inequality wherever $v_1(s)\neq v_2(s)$.  In particular,
$T_{\mathrm{arith}} < T_{\mathrm{harm}} = \frac{1}{2}(T_1+T_2)$, so the arithmetic mean
underestimates the period.

\begin{figure}[h]
\centering

\begin{subfigure}[b]{0.4\linewidth}
    \centering

    \begin{overpic}[width=\linewidth]{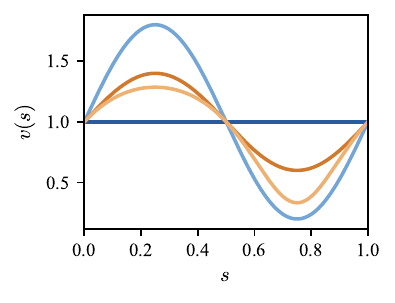}

        \put(63,58){
        \begin{minipage}{0.3\linewidth}
        \footnotesize

        \textcolor{vone}{\rule{0.8cm}{1.2pt}}~$v_1$\\[0.2em]
        \textcolor{vtwo}{\rule{0.8cm}{1.2pt}}~$v_2$\\[0.2em]
        \textcolor{varith}{\rule{0.8cm}{1.2pt}}~$v_{\mathrm{arith}}$\\[0.2em]
        \textcolor{vharm}{\rule{0.8cm}{1.2pt}}~$v_{\mathrm{harm}}$

\end{minipage}
        }

    \end{overpic}

    \caption{Speed profiles $v$}
\end{subfigure}
\hspace{0.02\textwidth}
\begin{subfigure}[b]{0.4\linewidth}
    \centering

    \begin{overpic}[width=\linewidth]{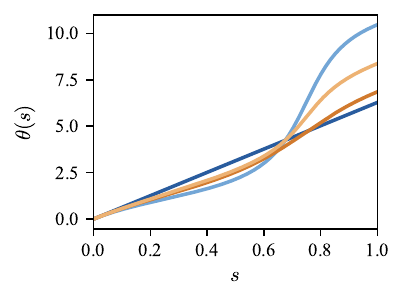}

        \put(25,56){
        \begin{minipage}{0.3\linewidth}
        \footnotesize

        \textcolor{vone}{\rule{0.8cm}{1.2pt}}~$\theta_1$\\[0.2em]
        \textcolor{vtwo}{\rule{0.8cm}{1.2pt}}~$\theta_2$\\[0.2em]
        \textcolor{varith}{\rule{0.8cm}{1.2pt}}~$\theta_{\mathrm{arith}}$\\[0.2em]
        \textcolor{vharm}{\rule{0.8cm}{1.2pt}}~$\theta_{\mathrm{harm}}$

\end{minipage}
        }

    \end{overpic}

    \caption{Cumulative time maps $\theta$}
\end{subfigure}

\caption{Comparison of arithmetic and harmonic averaging for dynamical reconstruction. (a) Speed profiles $v_1$ and $v_2$ together with their arithmetic mean $v_{\mathrm{arith}}$ and harmonic mean $v_{\mathrm{harm}}$. By Jensen's inequality, the harmonic mean lies below the arithmetic mean whenever the speed profiles differ. (b)~Corresponding cumulative traversal-time maps $\theta_1$ and $\theta_2$ together with the reconstructed averages $\theta_{\mathrm{arith}}$ and $\theta_{\mathrm{harm}}$. The harmonic reconstruction coincides with the pointwise average of the traversal times, whereas the arithmetic reconstruction underestimates the total traversal time and therefore the reconstructed period. This demonstrates that harmonic averaging provides the physically consistent reconstruction of the dynamics.}
\label{fig:speed_profiles_and_time_maps}
\end{figure}

\paragraph{Interpretation of $T_\mathrm{harm}$}
The equality in \eqref{eq:theta_harm} is specific to the present example,
in which all curves share the same arc length and geometric shape, and
serves to illustrate that the harmonic reconstruction can coincide with
the average of the sample periods. In the general case, where
the sample curves have different arc lengths $L_k$ and
$\|\mu_n'(s)\|\neq L_k$, this identity no longer holds. Nevertheless, the harmonic reconstruction retains a natural and
meaningful interpretation as an average of traversal times along the
mean curve. Indeed,
\[
  T_{\mathrm{harm}}
  = \int_0^1 \frac{\|\mu_n'(s)\|}{v_{\mathrm{harm}}(s)}\,
    \mathrm{d}s
  = \int_0^1 \|\mu_n'(s)\|
    \left(\frac{1}{N}\sum_{k=1}^N \frac{1}{\overline v_k^{\,\tau}(s)}\right)
    \mathrm{d}s
  = \frac{1}{N}\sum_{k=1}^N
    \int_0^1 \frac{\|\mu_n'(s)\|}{\overline v_k^{\,\tau}(s)}\, \mathrm{d}s,
\]
so that $T_{\mathrm{harm}}$ represents the average time required to
traverse the mean geometry using the individual speed profiles. In contrast, the arithmetic reconstruction averages speeds rather than
traversal times and therefore does not correspond to a consistent
averaging principle for the underlying dynamics. By Jensen's inequality,
\(
  T_{\mathrm{arith}} \leq T_{\mathrm{harm}},
\)
with strict inequality whenever the speed profiles differ. Thus, the
arithmetic reconstruction systematically underestimates the traversal
time, whereas the harmonic reconstruction yields a more faithful and
physically consistent representative of the sample dynamics.


\section{Assessing the Validity of Fr\'echet Means}\label{sec:diagnostics}

Even for simple samples in $\mathbb{R}^d$, the mean might not be a sensible summary, for instance
in the presence of multimodality or structural heterogeneity, a common phenomenon in statistical application
\cite{ronchetti2009robust,hastie2009elements}. This raises the question of when the
Fr\'echet mean constitutes a meaningful notion of central tendency and how to detect potential
failures. We propose two complementary classes of diagnostic tools:
curvature-based diagnostics, which assess the geometric regularity of the Fr\'echet mean, and medoid-based diagnostics, which assess its representativeness relative to the sample. Together, they provide a quantitative framework for assessing the validity of the empirical Fr\'echet mean.\footnote{Throughout this section, we consider only the geometric Fr\'echet mean obtained via arc length parametrization.}This question is particularly important in the heterogeneous dynamical regimes discussed in Section~\ref{sec:morris lecar model}.

\subsection{Curvature Diagnostics}\label{subsec:curvature_diag}
The Fréchet functional is defined with respect to the $L^2$ metric, which measures pointwise discrepancies between curves but does not control derivatives. Consequently, a Fréchet mean may be close to the sample in the $L^2$ sense while exhibiting substantially different geometric behavior, such as excessive bending or oscillatory structure. To assess geometric representativeness beyond the optimization criterion, we therefore introduce curvature-based diagnostics based on total curvature and bending energy. Curvature provides an intrinsic description of how a curve bends in state space, independently of
its parametrization \cite{tapp2016differential}. Let $\gamma \in C^2(\mathcal{S}^1,\mathbb{R}^d)$ be a twice continuously differentiable,
regular closed curve parametrized by normalized arc length. The curvature vector is
$\gamma''(s)$ and the associated scalar curvature is given by
\(
  \kappa(s) \;=\; \|\gamma''(s)\|.
\)

The total curvature of a closed curve is
\[
  \mathrm{TC}(\gamma) \;=\; \int_0^1 \kappa(s)\,\mathrm{d}s.
\]
A classical result of Fenchel~\cite{Fenchel1929} establishes the lower bound
\(
  \mathrm{TC}(\gamma) \;\geq\; 2\pi,
\)
with equality if and only if $\gamma$ is a planar convex curve. In the context of periodic trajectories, total curvature measures the cumulative turning per period
and is sensitive to the number and prominence of oscillatory or bursting events per period, serving
as an effective coarse diagnostic for detecting qualitatively different geometric regimes among the
sample curves \cite{tapp2016differential}. In heterogeneous samples, the distribution
$\{\mathrm{TC}(\gamma_k)\}_{k=1}^N$ may exhibit multiple well-separated clusters,
corresponding, for instance, to trajectories with differing numbers of bursts per period. A warning sign of non-representativeness is that the Fréchet mean attains a total curvature that lies between distinct clusters, or even outside the observed range.
In this regime, the Fr\'echet mean interpolates between incompatible patterns,
often resulting in a distorted non-representative trajectory. 

The bending energy of $\gamma$ is
\[
  E_{\mathrm{bend}}(\gamma) \;=\; \int_0^1 \kappa(s)^2\,\mathrm{d}s.
\]
Since any non-degenerate closed curve must bend ($\kappa \not\equiv 0$), we have
$E_{\mathrm{bend}}(\gamma) > 0$. This functional penalizes large curvature values quadratically and is particularly sensitive to sharp turns and fine-scale oscillations, in contrast to total curvature, which captures the overall amount of turning. In the theory of elastic curves, the total squared curvature serves as a measure of bending regularity and geometric complexity \cite{miura2025elasticcurvesselfintersections}. In heterogeneous datasets, the bending energy of the sample may exhibit multiple clusters corresponding to trajectories with differing smoothness or oscillatory complexity. A warning sign is when the Fréchet mean attains a bending energy outside these empirical regimes, suggesting interpolation between incompatible structures or the introduction of artificial fine-scale behavior.

\subsection{Dispersion and Representativeness Diagnostics}\label{subsec:disp_diag}

While the empirical Fr\'echet mean minimizes the empirical Fr\'echet functional, it is not
constrained to lie in the observed sample. In heterogeneous or multimodal datasets, the Fr\'echet
mean may therefore occupy a region of the curve space not well supported by any individual
trajectory. We introduce medoid-based diagnostics that quantify representativeness relative to the
empirical distribution. Similar representative objects are commonly used in cluster analysis, where medoids provide robust central elements of heterogeneous datasets \cite{kaufman2009finding}.

\paragraph{Fréchet medoid and mean--medoid distance}
Given a collection of curves $\Gamma_N = \{\gamma_1,\dots,\gamma_N\}$ in $\mathcal{C}$,
the Fr\'echet medoid is defined as any element
\[
  \mu_F^{\mathrm{med}}
  \;\in\;
  \arg\min_{\gamma_j \in \Gamma_N} \sum_{k=1}^N d(\gamma_j,\gamma_k)^2.
\]
We follow~\cite{bulté2023medoidsplitsefficientrandom} for the definition of the Fréchet medoid.
Unlike the Fr\'echet mean, which may yield a newly constructed curve, the Fr\'echet medoid is necessarily one of the observed trajectories.
It therefore represents the most central observed curve in the sample. In analogy with the empirical Fréchet variance, we define the Fréchet medoid variance as
\[
\mathrm{Var}_F^{\mathrm{med}}
=
\inf_{\gamma_j \in \Gamma_N}
\frac{1}{N}
\sum_{k=1}^N d(\gamma_j,\gamma_k)^2.
\]
Both the Fréchet variance and the Fréchet medoid variance measure dispersion of the sample, but with respect to different admissible sets, namely $\mathcal{C}$ and $\Gamma_N$. Since $\Gamma_N \subset \mathcal{C}$, it follows immediately that
\[
\mathrm{Var}_F^{(N)} \le \mathrm{Var}_F^{\mathrm{med}}.
\]
Equality holds if and only if at least one sample curve is an empirical Fréchet mean, that is, if the global minimizer over $\mathcal{C}$ can be chosen from $\Gamma_N$.

To measure how far the Fréchet mean deviates from the medoid, we consider
\(
d(\mu_F, \mu_F^{\mathrm{med}})^2.
\)
A small value indicates that the Fréchet mean resembles a typical observed curve, whereas a large value suggests that the mean lies in a region of the curve space that is not represented by a central sample trajectory. In a Hilbert space $\mathcal{H}$ the following variance decomposition holds for all $c \in \mathcal{H}$:
\begin{equation*}\label{eq:hilbert_var_decomp}
  \frac{1}{N}\sum_{k=1}^N \|c - \gamma_k\|_{L^2}^2
  \;=\; \|c - \bar{\gamma}\|_{L^2}^2 + \mathrm{Var}_F^{(N)},
\end{equation*} see Section~\ref{sec:frechet mean in metric spaces}.
Inserting $c = \mu_F^{\mathrm{med}}$ and
$\bar{\gamma} = \mu_F$ yields the exact identity
\begin{equation*}\label{eq:hilbert_medoid_gap}
  d\!\left(\mu_F,\,\mu_F^{\mathrm{med}}\right)^2
  \;=\; \mathrm{Var}_F^{\mathrm{med}} - \mathrm{Var}_F^{(N)} =: \Delta \mathrm{Var},
\end{equation*} and thus connecting the mean-medoid distance directly to the variance gap $\Delta \mathrm{Var} \geq 0$. In other words, the squared distance between mean and medoid is entirely determined by how much the restriction to observed curves costs in terms of the Fréchet functional. In particular, a large variance gap directly implies that the mean has moved away from any central observed trajectory.
This equality only holds in the Hilbert space setting. In the quotient space $\mathcal{C} = L^2(\mathcal{S}^1,\mathbb{R}^d)/\mathcal{S}^1$, the phase-alignment minimization in the metric $d$ breaks the linearity on which the Hilbert decomposition relies. However, we can still obtain a similar result, that also holds in the quotient space. Consider $\{\tau_k^*\}$ to be the optimal shifts of the sample to $\mu_F$, yielding the aligned curves $\tilde{\gamma}_k = \tau_k^* \cdot \gamma_k$. Then, the Hilbert identity holds exactly with 
\[\frac{1}{N}\sum_{k=1}^N \| c - \tilde{\gamma}_k \|_{L^2}^2 = \|c-\mu_F\|_{L^2}^2 + \mathrm{Var}_F^{(N)}\] for general $c \in \mathcal{H}.$ Now let $\tau^*$ be the optimal shift that aligns $\mu_F^\mathrm{med}$ to $\mu_F$. Choosing $c=\tau^* \cdot \mu_F^\mathrm{med}$ yields
\[d(\mu_F, \mu_F^\mathrm{med})^2 + \mathrm{Var}_F^{(N)} = \frac{1}{N}\sum_{k=1}^N \|\tau^* \cdot \mu_F^\mathrm{med} - \tilde{\gamma}_k \|_{L^2}^2.\]
Since $\tilde{\gamma}_k$ are optimally aligned to $\mu_F$, and not necessarily to $\mu_F^\mathrm{med}$, we obtain the following inequality
\[\|\tau^* \cdot \mu_F^\mathrm{med} - \tilde{\gamma}_k \|_{L^2}^2 \geq \min_\tau \|\tau^* \cdot \mu_F^\mathrm{med} - \tau \cdot \tilde{\gamma}_k \|_{L^2}^2 = d(\mu_F^\mathrm{med},\gamma_k)^2,\] using the $\mathcal{S}^1$-invariance of $d$. Combining these two equations yields
\begin{align*}
    d(\mu_F, \mu_F^\mathrm{med})^2 + \mathrm{Var}_F^{(N)} &= \frac{1}{N}\sum_{k=1}^N \|\tau^* \cdot \mu_F^\mathrm{med} - \tilde{\gamma}_k \|_{L^2}^2 \\
    &\geq \frac{1}{N} \sum_{k=1}^N \min_\tau \|\tau^* \cdot \mu_F^\mathrm{med} - \tau \cdot \tilde{\gamma}_k \|_{L^2}^2 \\ 
    &= \frac{1}{N} \sum_{k=1}^N d(\mu_F^\mathrm{med},\gamma_k)^2 = \mathrm{Var}_F^\mathrm{med},
\end{align*} resulting in a lower bound for $d(\mu_F,\mu_F^\mathrm{med})^2$ given by the variance gap $\Delta \mathrm{Var}$. This tightly couples the representativeness of the Fréchet mean in relation to the Fréchet medoid to the corresponding variances.
A large value of the variance gap, and thus also the mean--medoid distance, signals that the mean has moved away from any central observed trajectory. This deviation can have several causes, ranging from geometric heterogeneity to a skewed distribution or outliers present in the sample. In all cases, a large mean--medoid distance indicates that the Fréchet mean should be interpreted with caution, as it may not represent a typical trajectory of the dataset.

\paragraph{Nearest-neighbor distances and isolation ratio}
To assess local support of the Fr\'echet mean within the dataset, we introduce
\[
  \delta_F
  \;=\; \min_{1 \le i \le N} d\!\left(\mu_F, \gamma_i\right)^2,
  \qquad
  \delta_F^{\mathrm{med}}
  \;=\; \min_{\substack{1 \le i \le N \\ \gamma_i \ne \mu_F^{\mathrm{med}}}}
        d\!\left(\mu_F^{\mathrm{med}}, \gamma_i\right)^2.
\]
The quantity $\delta_F$ and $\delta_F^\mathrm{med}$ measure the squared distance from the Fréchet mean and medoid, respectively, to their nearest neighboring sample curve. 
Both quantities are bounded above by the respective Fr\'echet variances
\begin{equation}\label{eq:delta_bounds}
  \delta_F \;\le\; \mathrm{Var}_F^{(N)},
  \qquad
  \delta_F^{\mathrm{med}} \;\le\; \frac{N}{N-1}\,\mathrm{Var}_F^{\mathrm{med}},
\end{equation} where the factor in the second inequality stems from the exclusion of the medoid in the minimization. These inequalities show that neither
$\delta_F$ nor $\delta_F^{\mathrm{med}}$ can exceed the global scale of dispersion
of the respective representative, providing a consistency check between local and
global structure. The isolation ratio $\mathcal{S}$, defined by
\(
  \mathcal{S} \;=\; \delta_F/\delta_F^{\mathrm{med}},
\)
compares these two local scales: the proximity of the Fr\'echet mean to the dataset
against the intrinsic local support of the sample near the medoid. To interpret
$\mathcal{S}$, consider the geometric structure of the sample. In a geometrically coherent dataset, the Fréchet mean lies within the cloud of observed curves, so $\delta_F$ is of the same order as $\delta_F^{\mathrm{med}}$, 
, yielding $\mathcal{S}=O(1)$. In a heterogeneous sample, however, the Fréchet mean may be drawn toward an intermediate position between structurally distinct groups of curves. In this case, $\delta_F$, reflects inter-group separation rather than intra-group spacing, and $\mathcal{S}$ scales like the ratio of inter-group to intra-group squared distances, which may become large when groups are well separated relative to their internal spread. Thus, $\mathcal{S}$ can be viewed as a scale indicator: values of order one are consistent with geometric coherence, whereas $S\gg 1$ suggests substantial isolation of the Fréchet mean relative to the sample's intrinsic local spacing.

The proposed diagnostics assess complementary aspects of representativeness and should be interpreted jointly. Curvature-based quantities detect geometric artifacts that may arise from interpolation between incompatible trajectory structures, while the mean--medoid distance and variance gap quantify global displacement of the Fr\'echet mean relative to central observed trajectories. The isolation ratio further measures whether the mean remains locally supported by the sample. Together, these diagnostics distinguish between loss of geometric regularity, displacement from the data cloud, and lack of local support, thereby identifying different mechanisms by which the empirical Fr\'echet mean may cease to provide a meaningful summary of the sample.

\subsection{Illustrative Example}

To illustrate the behavior of the proposed diagnostics, we consider a synthetic dataset consisting of two geometrically distinct classes of closed curves, consisting of circular and figure-eight curves, inducing a bimodal distribution characterised by distinct geometric complexity. Figure~\ref{fig:diagnostics_samples_and_curves} shows the sample separated by classes, together with the resulting geometric Fréchet mean and medoid. While each class is internally coherent, the combined dataset displays pronounced geometric heterogeneity. The Fréchet mean interpolates between the two classes and produces a curve that does not resemble any observed trajectory. In particular, while the sample curves in the second class exhibit a symmetric bilobed structure, the Fréchet mean yields a distorted, asymmetric version of such a configuration. This introduces artificial asymmetry not present in the data and fails to capture the characteristic geometric patterns from either class. Consequently, the mean does not represent a meaningful representative of the sample.
\begin{figure}[h]
\centering

\begin{subfigure}[b]{0.28\linewidth}
    \centering
    \includegraphics[width=\linewidth]{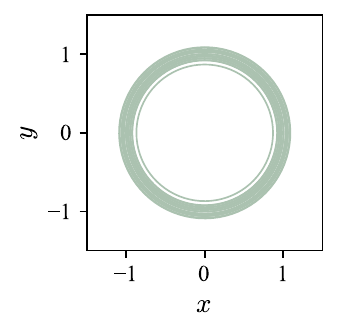}
    \caption{Circular sample}
\end{subfigure}
\begin{subfigure}[b]{0.28\linewidth}
    \centering
    \includegraphics[width=\linewidth]{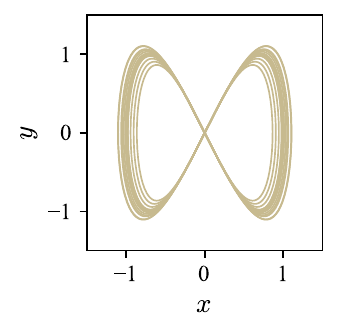}
    \caption{Self-intersecting sample}
\end{subfigure}
\begin{subfigure}[b]{0.28\linewidth}
    \centering
    \includegraphics[width=\linewidth]{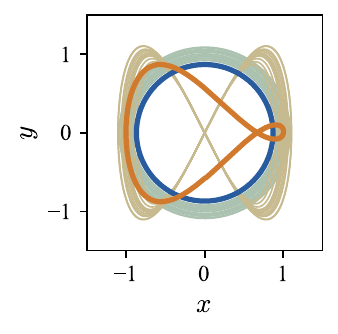}
    \caption{Fréchet mean and medoid}
\end{subfigure}
\caption{Illustration of a failure mode of the Fr\'echet mean in a geometrically heterogeneous dataset. (a) Sample of circular curves with small radial variability. (b) Sample of symmetric self-intersecting figure-eight curves. (c) Fr\'echet medoid (blue) and Fr\'echet mean (orange) for the combined dataset. While the medoid remains representative of one observed class, the Fr\'echet mean interpolates between incompatible geometric structures and produces an asymmetric distorted curve not present in the data, indicating a loss of representativeness.}
\label{fig:diagnostics_samples_and_curves}
\end{figure}

This lack of representativeness is further reflected in the curvature-based diagnostics. As shown in Figure~\ref{fig:example_curvature_diagnostics}, both the total curvature and the bending energy of the Fréchet mean lie outside the clusters formed by the sample curves. In fact, the values for the Fréchet mean are substantially larger than those observed in either class, indicating that the Fréchet mean introduces additional geometric complexity---manifested as excessive bending and irregularity---that is not present in the data.
\begin{figure}[h]
\centering

\begin{subfigure}[b]{0.35\linewidth}
    \centering
    \includegraphics[width=\linewidth]{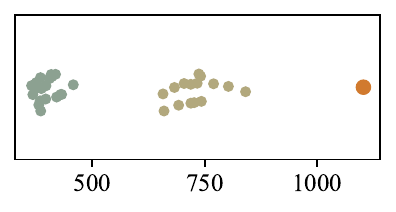}
    \caption{Total curvature}
\end{subfigure}
\begin{subfigure}[b]{0.35\linewidth}
    \centering
    \includegraphics[width=\linewidth]{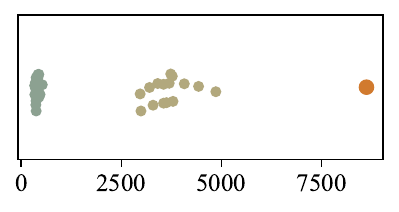}
    \caption{Bending energy}
\end{subfigure}
\caption{Curvature-based diagnostics for the synthetic bimodal dataset. The distributions of total curvature (a) and bending energy (b) form two well-separated clusters corresponding to the circular and self-intersecting curve classes. In both quantities, the Fr\'echet mean lies outside the empirical range and attains substantially larger values, indicating the introduction of artificial geometric irregularities and oscillatory structure not present in the sample trajectories.}
\label{fig:example_curvature_diagnostics}
\end{figure}
The medoid-based diagnostics further confirm this behavior, see Table~\ref{tab:example_diagnostics}. The mean--medoid distance is large relative to the overall spread of the dataset, as quantified by the Fréchet variances, indicating that the mean has moved far away from 
any centrally located observed curve. Moreover, the nearest-neighbor distance $\delta_F$ of the mean is significantly larger than $\delta_F^\mathrm{med}$, resulting in a significantly elevated isolation ratio. These observations indicate that the Fréchet mean is not only globally displaced from the central region of the data, but also locally unsupported by nearby sample curves.

\begin{table}[t]
\centering
\caption{Summary of diagnostic quantities for the synthetic bimodal dataset. The large variance gap $\Delta \mathrm{Var} = \mathrm{Var}_F^{\mathrm{med}} - \mathrm{Var}_F$, together with the substantial mean--medoid distance and the high isolation ratio $\mathcal{S}$, indicate that the Fréchet mean is both globally displaced and locally unsupported.}
\label{tab:example_diagnostics}
\begin{tabular}{cccccc}
    \toprule
    $\mathrm{Var}_F$ & $\mathrm{Var}_F^{\mathrm{med}}$ & $d(\mu_F, \mu_F^{\mathrm{med}})^2$ & $\delta_F$ & $\delta_F^{\mathrm{med}}$ & $\mathcal{S}$ \\
    \midrule
    0.244 & 0.376 & 0.132 & 0.132 & 0.002 & 82.880\\
    \bottomrule
\end{tabular}
\end{table}

Overall, this example demonstrates a clear failure mode of the Fréchet mean in the presence of geometric heterogeneity. The Fréchet mean introduces artificial asymmetry and exhibits excessive curvature and bending beyond the observed range. The proposed diagnostics consistently detect these effects and provide a principled means of identifying the mechanisms responsible for the breakdown of the Fréchet mean as a representative trajectory.

\section{The Van der Pol Oscillator}\label{sec:vdp}

As a first illustrative example for the computation of geometric and dynamical Fréchet means, we consider the two–dimensional van der Pol oscillator. This system is a prototypical model for nonlinear self-sustained oscillations with a pronounced fast-slow structure and a stable limit cycle. 

\paragraph{The Fast-Slow Structure and Periodic Orbits}
We consider the van der Pol system in the following fast–slow form~\cite{kuehn2015multiple}
\begin{align*}
    \dot{x} &= y - \frac{x^3}{3} + x, \\
    \dot{y} &= \varepsilon (a - x).
\end{align*} The variable $0 < \varepsilon \ll 1$ controls the separation of time scales and we fix it to be $\varepsilon = 0.1$. For parameter values $a \in (-1,1)$, the system possesses a unique attracting limit cycle. The only equilibrium point in this regime is given by the unstable steady state
\(
(x^*,y^*) = (a,a^3/3-a).
\)
One can prove in this regime that for every initial condition except the equilibrium itself, the solution converges to the periodic orbit. After a transient phase, the long–term dynamics are therefore completely described by this attracting limit cycle, see Figure~\ref{fig:vdp_solution}.
\begin{figure}[htbp]
\centering

\begin{subfigure}[t]{0.55\textwidth}
    \vspace{0pt}
    \centering
    \includegraphics[width=\linewidth]
    {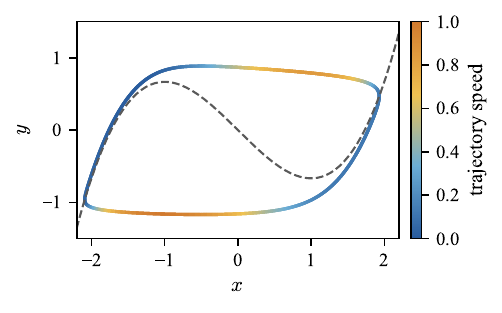}
    \caption{Phase portrait}
    \label{fig:vdp_phaseportrait}
\end{subfigure}
\hspace{0.02\textwidth}
\begin{subfigure}[t]{0.33\textwidth}
    \vspace{0pt}
    \centering
    \includegraphics[
        width=\linewidth,
        height=.38\textheight,
        keepaspectratio
    ]{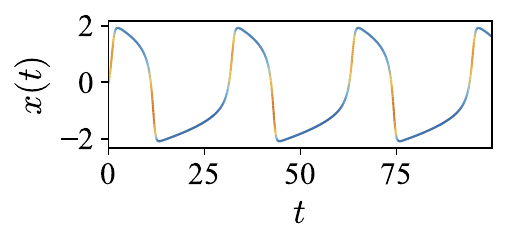}

    \includegraphics[
        width=\linewidth,
        height=.38\textheight,
        keepaspectratio
    ]{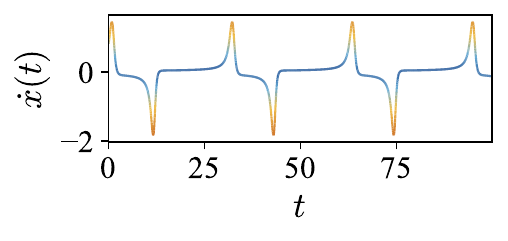}
    \vspace{-0.4pt}
    \caption{State coordinates $x$ and $\dot{x}$}
    \label{fig:vdp_timeseries}
\end{subfigure}

\caption{Dynamics of the van der Pol oscillator in the fast--slow regime with $\varepsilon = 0.1$. (a) Phase portrait showing attracting limit cycle. The trajectory exhibits slow motion along the attracting branches of the cubic critical manifold and rapid transitions between them, producing the characteristic relaxation oscillation structure. (b) Time series of the state variable $x$ and its derivative $\dot{x}$ along the periodic orbit, highlighting the alternation between slow and fast dynamical phases.}
\label{fig:vdp_solution}
\end{figure}

For small $\varepsilon$, the system separates into slow and fast dynamics: the variable $x$ changes quickly, while $y$ evolves slowly. In the limit $\varepsilon \to 0$, the fast dynamics relaxes rapidly onto the set where $\dot{x}=0$, which is given by
\[y = \frac{x^3}{3} - x.\]
This cubic curve is called the critical manifold.
Trajectories move slowly near its attracting outer branches and undergo rapid transitions between them when stability is lost. Consequently, the periodic orbit consists of alternating segments of slow motion along the cubic and fast horizontal jumps between the branches. This geometric decomposition into slow drift along the critical manifold and fast transitions between branches makes the van der Pol oscillator an ideal benchmark for studying parametrization effects in Fréchet mean computations. Furthermore, we analyze how a decoupled reconstruction of the dynamics under arc length parametrization can be used to obtain a dynamical Fréchet mean.

\paragraph{Computation of the Geometric Fréchet Mean}
We now describe the computational pipeline used to construct empirical Fréchet means from a collection of sample trajectories of the van der Pol system. We consider $N$ sample trajectories corresponding to parameter values 
$a_1,\dots,a_N \in (-1,1).$ For each parameter value $a_k$, the system is solved numerically using a standard ODE solver, yielding a trajectory parametrized by physical time $t$. To isolate the periodic behavior, we discard an initial transient phase. From the remaining trajectory, we extract exactly one period by detecting two consecutive local maxima of the $x$-component. The portion of the trajectory between these maxima defines a closed curve
\(
\gamma_k : [t_k, t_k + T_k] \to \mathbb{R}^2,
\)
where $t_k$ denotes the starting time and $T_k$ the period of the $k$-th trajectory. This procedure produces a collection of closed curves $\gamma_k$, each associated with a different parameter value $a_k$. Each extracted periodic orbit is naturally parametrized by physical time, as determined by the van der Pol dynamics. We consider two normalization strategies: time normalization and arc length reparametrization, denoted by $\tilde{\gamma}_k, \ \overline{\gamma}_k : \mathcal{S}^1\to \mathbb{R}^2$, see Section~\ref{sec:a metric structure for closed curves}. Given this collection of parametrized curves, we compute the empirical Fréchet mean using the iterative alignment algorithm described in Section~\ref{sec:computation of the empirical frechet mean}. The resulting mean curve provides a numerical approximation of the empirical Fréchet mean of the sample. The entire procedure is applied separately to the time-normalized curves $\tilde{\gamma}_k$ and the arc length parametrized curves $\overline{\gamma}_k$. This allows us to quantify the influence of parametrization on the resulting mean curve.

\paragraph{Parametrization Effects on the Geometric Fréchet Mean} 
The sample curves and the corresponding Fréchet mean approximations for both parametrizations are shown in Figure~\ref{fig: frechet mean vdp}.
\begin{figure}[h]
\centering

\begin{subfigure}[t]{0.4\textwidth}
    \vspace{0pt}
    \centering
    \includegraphics[width=\linewidth]
    {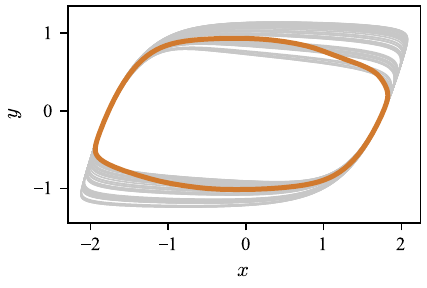}
    \caption{Time parametrization}
\end{subfigure}
\begin{subfigure}[t]{0.4\textwidth}
    \vspace{0pt}
    \centering
    \includegraphics[
        width=\linewidth,
        height=.38\textheight,
        keepaspectratio
    ]{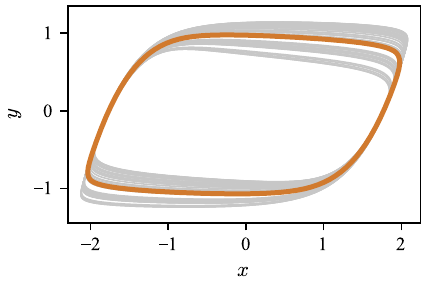}
    \caption{Arc length parametrization}
\end{subfigure}
\caption{The geometric Fr\'echet mean computed from $N=50$ periodic trajectories of the van der Pol system corresponding to parameter values $a_1,\dots,a_N\in(-1,1)$ (grey curves), using (a) time parametrization and (b) arc length parametrization. In the time-parametrized setting, the resulting Fr\'echet mean exhibits noticeable geometric distortions caused by the underlying fast--slow dynamics of the system. Since trajectories spend substantially more time in slow regions than in fast transition regions, the sampling density along the orbit becomes highly non-uniform. In addition, small differences in the temporal dynamics between trajectories lead to local phase mismatches that further distort the averaged curve. As a result, the time-parametrized mean no longer preserves the characteristic trapezoidal relaxation-oscillation geometry of the sample trajectories. In contrast, the arc length parametrization removes the influence of traversal speed and considers only the geometric structure of the curves. The resulting Fr\'echet mean therefore preserves the characteristic geometric shape of the van der Pol limit cycles while remaining centered within the family of sample trajectories.}
\label{fig: frechet mean vdp}
\end{figure}
The grey curves represent $N=50$ trajectories generated from parameter values $a_1,\dots,a_N \in (-1,1)$ drawn uniformly at random\footnote{A fixed random seed (123) is used to ensure reproducibility. The same set of trajectories is used for both the time-normalized and arc length parametrized computations.}. For all parameter values in $(-1,1)$, the periodic orbits exhibit a characteristic trapezoidal shape induced by the fast–slow structure of the dynamics.
Under time normalization, the sampling density along each curve reflects the underlying dynamics: many points lie on the slow segments, while comparatively few points lie on the fast transitions. Since the Fréchet mean is computed with respect to the $L^2$-metric on aligned parametrized curves, this non-uniform sampling implicitly assigns greater weight to the slow regions. As a consequence, the alignment and averaging steps mix slow segments of some trajectories with fast segments of others, leading to a geometrically rounded mean curve that does not preserve the trapezoidal structure of the individual orbits. In contrast, the arc length parametrization removes the influence of dynamical speed and distributes sampling points uniformly with respect to geometric distance along the curve. The Fréchet mean computed in this representation therefore reflects the geometric shape of the periodic orbits rather than their time-dependent traversal. As a result, the mean curve preserves the characteristic trapezoidal structure of the sample trajectories. This comparison illustrates that, for fast–slow systems such as the van der Pol oscillator, the choice of parametrization has a significant impact on the geometric Fréchet mean. In particular, arc length parametrization yields a Fréchet mean that reflects the geometric structure of the family of closed orbits.

\paragraph{Dynamical Reconstruction of the Fréchet Mean}
In the arc length formulation, trajectories are treated as purely geometric objects, so the temporal information associated with non-uniform traversal along the orbit is removed. In particular, the characteristic fast--slow structure of the van der Pol system is not encoded in the geometric Fr\'echet mean alone. To reconstruct representative dynamics, we follow the procedure introduced in Section~\ref{sec:reconstructing dynamics}. Using the harmonic mean approach, we obtain the mean speed $v_{\mathrm{harm}}(s)$ and the mean cumulative time map $\theta_{\mathrm{harm}}(s)$, describing the dynamic information of the Fréchet mean curve. In particular, the inverse of $\theta_{\mathrm{harm}}(s)$ parametrizes the geometric mean and produces the dynamically reconstructed trajectory, see Figure~\ref{fig: dynamics of vdp}. The dynamic Fréchet mean sensibly recovers the characteristic fast--slow behavior of the van der Pol system---slow motion along the attracting branches and rapid transitions between them---while remaining geometrically centered within the family of sample trajectories.
\begin{figure}[htbp]
\centering

\begin{subfigure}[t]{0.55\textwidth}
    \vspace{0pt}
    \centering
    \includegraphics[width=\linewidth]
    {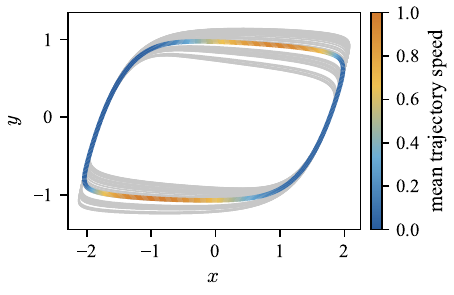}
    \caption{Dynamical Fréchet mean}
\end{subfigure}
\hspace{0.02\textwidth} 
\begin{subfigure}[t]{0.35\textwidth}
    \vspace{0.2em}
    \centering
    \includegraphics[
        width=\linewidth,
        height=.38\textheight,
        keepaspectratio
    ]{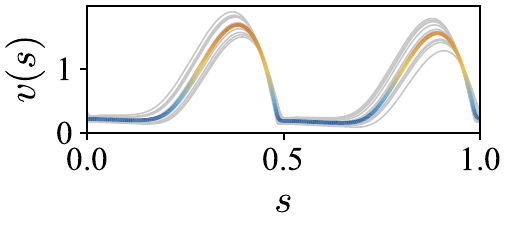}

    \vspace{0.5em}

    \includegraphics[
        width=\linewidth,
        height=.38\textheight,
        keepaspectratio
    ]{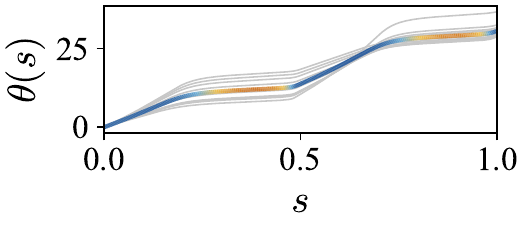}
    \vspace{-0.3em}
    \caption{Speed and traversal time}
\end{subfigure}

\caption{Dynamical reconstruction of the geometric Fr\'echet mean for the van der Pol system. (a)~ The geometric Fr\'echet mean together with the reconstructed dynamics shown through the color-coded trajectory speed. The reconstruction combines the geometrically averaged curve obtained from arc length parametrization with dynamics reconstructed from the aligned speed profiles using harmonic averaging. The resulting trajectory recovers the characteristic fast--slow structure of the van der Pol oscillator, exhibiting slow motion along the attracting branches and rapid transitions between them. (b) Corresponding aligned speed profiles $v$ and cumulative traversal-time maps $\theta$ for the sample trajectories together with their reconstructed averages. Both quantities interpolate naturally between the sample curves, indicating that the reconstructed dynamics provide a representative summary of the temporal behavior of the trajectory family. An animation illustrating the dynamics is available at \url{https://github.com/paulinash/FrechetMeanPaper}.}
\label{fig: dynamics of vdp}
\end{figure}

\section{Rosenzweig-MacArthur Model}\label{sec:rosenzweig-macarthur model}

In this section, we consider the classical predator--prey system introduced by \cite{rosenzweig1963graphical}. The model consists of a two-dimensional dynamical system that exhibits stable limit cycles arising through a Hopf bifurcation. We apply the proposed Fréchet mean ansatz to families of periodic solutions obtained via parameter variation and investigate whether the resulting mean curve yields a geometrically and dynamically meaningful representative trajectory. In particular, we examine how different parameter perturbations influence the resulting Fréchet mean.

\paragraph{The Predator-Prey model}
The Rosenzweig-MacArthur model is given in the following form 

\begin{align}
    \frac{d N}{d T} &= r N\left(1-\frac{N}{K}\right) - \phi(N)P,\\[2mm]
    \frac{d P}{d T} &= \frac{b N P}{a + N} - m P,
\end{align}
with a type-II functional response 
\(\phi(N) = \frac{c N}{a+N}\)~\cite{kot2001elements}. Here, \(N\) and \(P\) denote the prey and predator population sizes, respectively. 
The parameters have the following biological interpretations: \(r\) is the intrinsic growth rate of the prey, 
\(K\) is the prey carrying capacity, \(a\) is the predator half-saturation constant, \(b\) is the conversion rate from prey to predator, \(c\) is the maximal predation rate, and \(m\) is the predator mortality rate. For appropriate parameter regimes, the system admits stable limit cycles.
The functional response $\phi(N)$ models the rate at which each predator captures prey. The commonly used type II functional response is a hyperbolic function that saturates due to the time it takes to handle the prey.
Applying the change of variables
\(
N = a x,  P = r y a/c,  T = t/r,
\)
the system can be rewritten as
\begin{align}
\frac{dx}{dt} &= x\left(1 - \frac{x}{\gamma}\right) - \frac{x y}{1 + x}, \\[1mm]
\frac{dy}{dt} &= \beta \left(\frac{x}{1 + x} - \alpha\right) y,
\end{align}
with
\(
\alpha = m/b,  
\beta = b/r, 
\gamma = K/a.
\)
The parameter $\alpha$ is the predator loss-to-gain ratio, $\beta$  compares the predator growth to the intrinsic prey growth and $\gamma$ is the enrichment parameter
To study the effect of parameter variation while ensuring the presence of periodic solutions, the parameters need to satisfy
\[
0 < \alpha < 1, \quad \beta > 0, \quad \gamma > \gamma_H = \frac{1 + \alpha}{1 - \alpha},
\]
where \(\gamma_H\) is the Hopf bifurcation threshold.  
Since \(\gamma_H\) depends on \(\alpha\), it is convenient to define a dimensionless multiplier \(\gamma_\mathrm{mult} > 1\) and set
\[
\gamma = \gamma_H \, \gamma_\mathrm{mult} = \frac{1+\alpha}{1-\alpha}\,\gamma_\mathrm{mult}.
\]
This formulation allows independent variation of the three parameters \(\alpha, \beta,\) and \(\gamma_\mathrm{mult}\) while remaining in the oscillatory regime. The existence of stable periodic solutions makes the system a suitable test case for the geometric Fréchet mean framework introduced above. In particular, it allows us to investigate how parameter-induced variability of limit cycles is reflected in the corresponding Fréchet mean.

\subsection{Analysis of Fréchet Mean}
We now analyze the Fréchet mean,  
focusing on the influence of the parameters $\alpha$, $\beta$, and $\gamma_{\mathrm{mult}}$. The computational workflow follows that of Section~\ref{sec:vdp}. We investigate three distinct cases of parameter variation: the first two involve varying only $\alpha$ and $\gamma_{\mathrm{mult}}$ respectively, while the third considers simultaneous variation of all three parameters $\alpha$, $\beta$, and $\gamma_{\mathrm{mult}}$.
We restrict the parameter ranges to moderate intervals in order to avoid large-amplitude oscillations and slow transient convergence near the bifurcation boundary.

\paragraph{Variation of $\alpha$}
We randomly sample 12 values of  $\alpha \in (0.2,0.6)$ while keeping $\beta = 0.752$ and $\gamma_\mathrm{mult}=1.76$ fixed. The resulting periodic orbits exhibit the characteristic predator–prey phase portrait: a progression along the $x$-axis (small predator density), followed by a rapid increase in $y$, which results in a decrease in $x$. After reaching a minimal prey density, the predator population decreases, and the trajectory returns toward the origin before the next oscillation cycle begins, see Figure~\ref{fig:rm_phase}. The variation in $\alpha$ mainly affects the maximal amplitude in $x$- and $y$-direction---i.e. the maximal prey and predator population. Geometrically, this leads to sample curves that are increasingly stretched in the rightward and upward directions in phase space, while the segments near the coordinate axes remain largely unchanged. The Fréchet mean curve captures this deformation in a balanced manner: it preserves the common geometric features near the coordinate axes, where all trajectories nearly coincide, while interpolating between the varying peak amplitudes and the local curvature profile at the prey maximum. A notable effect occurs near the origin. Although the sample curves all exhibit a sharper change in direction, the mean curve appears slightly smoother. This may result from the use of constant phase shifts in the alignment procedure, where small local misalignments can introduce mild smoothing during averaging. In this example, the effect is subtle, though more flexible alignment methods such as the square-root velocity function (SRVF) framework could potentially preserve local geometric features more accurately~\cite{srivastava2016functional}. 
\begin{figure}[h]
\centering

\begin{subfigure}[t]{0.45\textwidth}
    \vspace{0pt}
    \centering

    \includegraphics[width=\linewidth]
    {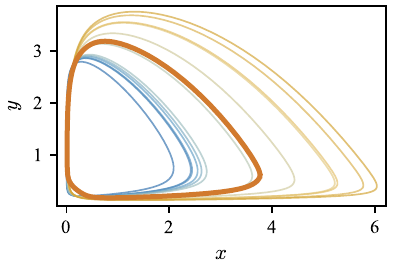}
    \caption{Geometric Fréchet mean}
    \label{fig:rm_phase}
\end{subfigure}
\hspace{0.02\textwidth}
\begin{subfigure}[t]{0.31\textwidth}
    \vspace{0pt}
    \centering
    \includegraphics[width=\linewidth,height=.42\textheight,keepaspectratio]
    {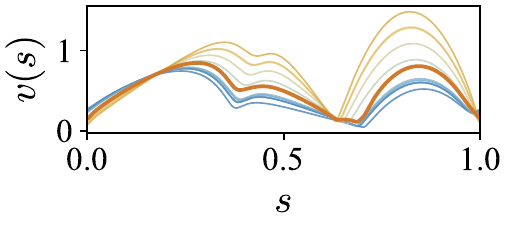}

    \includegraphics[width=\linewidth,height=.42\textheight,keepaspectratio]
    {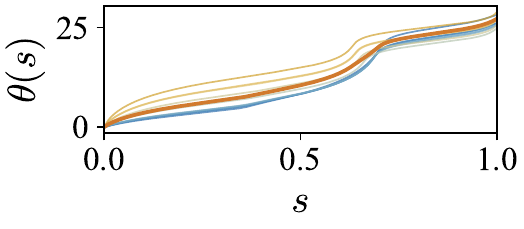}
    \vspace{0pt}
    \caption{Speed $v$ and time map $\theta$}
    \label{fig:rm_dynamic_information}

\end{subfigure}

\caption{Fr\'echet mean for the Rosenzweig--MacArthur model under variation of the predator loss-to-gain ratio $\alpha\in(0.2,0.6)$. (a) Limit cycles corresponding to varying values of $\alpha$ are shown in blue and yellow, together with the geometric Fr\'echet mean in orange. The sample trajectories exhibit the characteristic predator--prey oscillatory structure, while the Fr\'echet mean captures this geometry in an intermediate and geometrically consistent manner. In particular, the mean interpolates between variations in amplitude and curvature while preserving the common phase-space structure of the trajectories. (b) Corresponding aligned speed profiles $v$ and cumulative traversal-time maps $\theta$. The reconstructed dynamical averages $v_{\mathrm{harm}}$ and $\theta_{\mathrm{harm}}$ remain closest to the grey-blue sample trajectory, which is also geometrically closest to the Fr\'echet mean in phase space. This simultaneous agreement in geometry and dynamics demonstrates that the Fr\'echet mean is capable of providing a physically meaningful representative trajectory for the family of oscillations.}

\label{fig:rm_alpha}
\end{figure}

In Figure~\ref{fig:rm_dynamic_information}, we additionally display the aligned speed profiles $v$ and cumulative time maps $\theta$ associated with the sample trajectories together with their reconstructed averages along the Fr\'echet mean curve. The averaged quantities preserve the characteristic temporal structure of the system, in particular the distinction between slow and fast phases. In the speed profile $v$, one additionally observes adjacent local minima of the mean curve. This feature appears as an averaging artifact caused by slight shifts in the locations of the corresponding single minima across the sample trajectories. While the effect remains subtle in the present example, it illustrates a limitation of the dynamic averaging procedure: localized temporal features may become broadened or duplicated under alignment and averaging, and should therefore be interpreted with some care. Moreover, the reconstructed profiles along the mean curve interpolate naturally between those of the individual sample trajectories and are most similar to trajectories that are geometrically closest to the Fr\'echet mean in phase space (blue grey orbit). This consistency between geometric proximity and reconstructed temporal behavior supports the interpretation of the Fr\'echet mean as a representative trajectory combining both geometric and dynamical information.

\paragraph{Variation of $\gamma_\mathrm{mult}$}
We vary $\gamma_\mathrm{mult} \in (1.1,2)$ using 10 samples, while fixing $\alpha = 0.479$ and $\beta=0.644$, see Figure~\ref{fig:rm_gamma}. In contrast to varying $\alpha$, variation of $\gamma_\mathrm{mult}$ primarily induces a scaling effect on the limit cycles. The trajectories retain the same qualitative geometric structure, while their overall amplitude increases with $\gamma_\mathrm{mult}$. In addition, the orbits are no longer confined to run closely along the coordinate axes. Instead, their minimal $x$- and $y$-values increase, resulting in a gradual displacement of the cycles away from the axes. The geometric Fréchet mean captures both effects coherently. The mean curve preserves the characteristic oscillatory geometry while exhibiting an intermediate amplitude and an intermediate phase-space location. The scaling behavior is reflected consistently along the entire orbit, including the segments near the coordinate axes, indicating that the averaging procedure respects the global geometric structure rather than distorting individual regions.

\begin{figure}[h]
\centering

\begin{subfigure}[t]{0.45\textwidth}
    \vspace{0pt}
    \centering

    \includegraphics[width=\linewidth]
    {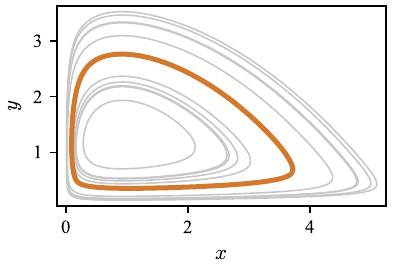}
    \caption{Variation of $\gamma_\mathrm{mult}$}
    \label{fig:rm_gamma}
\end{subfigure}
\hspace{0.02\textwidth}
\begin{subfigure}[t]{0.45\textwidth}
    \vspace{0pt}
    \centering
    \includegraphics[width=\linewidth]
    {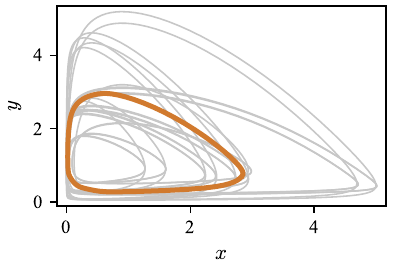}
    \caption{Variation of parameters $\alpha, \beta$ and $\gamma_\mathrm{mult}$}
    \label{fig:rm_all}

\end{subfigure}
\caption{Geometric Fr\'echet means for the Rosenzweig--MacArthur model under different parameter variations. (a) Variation of the parameter $\gamma_{\mathrm{mult}}$ primarily induces a global scaling of the limit cycles together with a gradual displacement away from the coordinate axes. Although the amplitudes and phase-space locations vary across the sample, the Fr\'echet mean captures these scaling effects coherently while preserving the qualitative predator--prey oscillation geometry. (b) Simultaneous variation of the parameters $\alpha$, $\beta$, and $\gamma_{\mathrm{mult}}$ generates substantially increased geometric variability, including changes in amplitude, scaling, curvature, and phase-space location. Despite this increased variability, the geometric Fr\'echet mean yields a coherent representative trajectory that preserves the characteristic oscillatory predator--prey structure of the sample trajectories.}
\label{fig: rm_gamma_and_all}
\end{figure}

\paragraph{Simultaneous Variation of $\alpha, \beta$, and $\gamma_\mathrm{mult}$}
We sample
\(
\alpha \in (0.2,0.6), 
\beta \in (0.1,2),\) and \( 
\gamma_\mathrm{mult} \in (1.1,2)
\)
using three values per parameter, resulting in $3^3 = 27$ sample trajectories, see Figure~\ref{fig:rm_all}. In contrast to varying $\alpha$ or $\gamma_\mathrm{mult}$ separately, as done before, this setting combines several geometric effects simultaneously, including amplitude, scaling, phase-shift, and curvature effects.
Consequently, the family of limit cycles displays substantially increased geometric variability. Despite this increased variability, the Fréchet mean remains well-defined and produces a coherent representative closed curve. It preserves the characteristic predator–prey oscillatory structure and reflects an intermediate amplitude and curvature profile while remaining geometrically consistent with the family of sample trajectories. Overall, the Rosenzweig--MacArthur example illustrates that the Fr\'echet mean preserves both the qualitative phase-space geometry and the dynamical structure of the oscillations across different types of parameter-induced variability, including amplitude changes, scaling and translation effects.

\section{Morris-Lecar Model}\label{sec:morris lecar model}
In this section, we investigate the Fréchet mean for the Morris–Lecar model, a three-dimensional neuronal system~\cite{kuehn2015multiple}. Its solutions are periodic but exhibit bursting behavior, with trajectories differing in the number of bursts per period. This variability substantially affects the geometric structure of the corresponding curves. Thus, we examine the effects of the Fréchet mean for the homogeneous case, i.e. all sample curves display an equal number of bursts per period, and heterogeneous case, i.e. the sample curves exhibit distinct number of bursts. We assess these cases using the curvature and Fréchet medoid diagnostics proposed in Section~\ref{sec:diagnostics}. In contrast to the previous two-dimensional examples, the studied Morris-Lecar system is three-dimensional, illustrating that the proposed framework  naturally extends to higher-dimensional settings.

\paragraph{The Bursting Model}
The Morris-Lecar model serves as a prototypical example of bursting oscillations, a phenomenon  originally studied in the context of neuroscience. Bursting oscillations are characterized by periodic time series patterns that are alternating between near steady-state and rapid oscillatory phases. In the Morris-Lecar model, bursting arises from the interaction of global and local bifurcations in a fast-slow system. More precisely, the Morris-Lecar model can be interpreted as a $(2,1)-$fast-slow system, with two fast variables and one slow variable. The bursting phenomenon manifests as multiple oscillations---typically three to five---within one global period, approaching a limit cycle in the fast subsystem, see Figure~\ref{fig: phase portrait ml}. 
\begin{figure}[h]
\centering
\begin{subfigure}[t]{0.4\textwidth}
    \vspace{1.2cm}
    \centering

    \includegraphics[width=\linewidth]
    {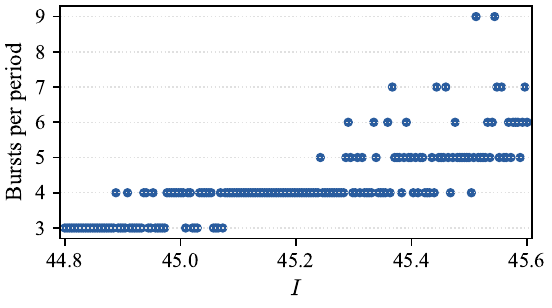}
    \vspace{0.4cm}
    \caption{Burst count}
    \label{fig: number_of_bursts}

\end{subfigure}
\hspace{0.02\textwidth}
\begin{subfigure}[t]{0.45\textwidth}
    \vspace{0pt}
    \centering
    \includegraphics[width=\linewidth,height=.42\textheight,keepaspectratio]
    {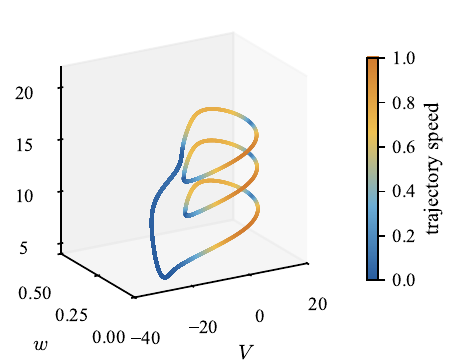}
    \caption{Phase portrait}
    \label{fig: phase portrait ml}
\end{subfigure}
\caption{Bursting dynamics in the Morris--Lecar system. (a) Number of fast oscillations per global bursting period as a function of the input current $I$, computed from $200$ sampled trajectories with $I\in[44.8,45.6]$. The plot illustrates how the burst count depends on the parameter value and reveals regions of stable oscillation numbers together with transition regions in which the number of spikes changes irregularly and the distribution becomes increasingly chaotic. This parameter regime therefore provides a controlled setting for studying geometric heterogeneity in periodic trajectories. (b) Representative periodic bursting orbit for $I=44.9$, exhibiting three bursts within one global period. The color grading distinguishes slow motion along the slow manifold (blue) from fast transitions in the burst segments (orange), highlighting the underlying fast--slow structure of the dynamics.}
\label{fig: ml solutions and number of bursts}
\end{figure}
Rather than focusing on the dynamical mechanisms of bursting, we use the variation in burst number as a structured example of geometric heterogeneity. This setting allows us to assess when the Fréchet mean provides a meaningful representative trajectory and when it fails to reflect the underlying geometric structure. The Morris-Lecar Model is considered using the following formulation
\begin{align*}
c\,\frac{dx_1}{dt}&=I - I_{\mathrm{ca}}- \left(g_{\mathrm{k}} x_2 + g_{\mathrm{kca}}\,\frac{y}{y+y_0}\right)(x_1 - V_k)- g_l (x_1 -V_l), \\
\frac{dx_2}{dt}&=\phi\,\tau_w(x_1)\,\bigl(w_\infty(x_1) - x_2\bigr), \\
\frac{dy}{dt}&=\varepsilon\,\bigl(-\mu\, I_{\mathrm{ca}} - y\bigr).
\end{align*} with auxiliary functions
\begin{align*}
m_\infty(x_1)&= \tfrac{1}{2}\left(1 + \tanh\!\left(\frac{x_1 - V_1}{V_2}\right)\right), 
& \tau_w(x_1)&= \cosh\!\left(\frac{x_1 - V_3}{2V_4}\right), \\
w_\infty(x_1)&= \tfrac{1}{2}\left(1 + \tanh\!\left(\frac{x_1 - V_3}{V_4}\right)\right),
&I_{\mathrm{Ca}}&= g_{\mathrm{Ca}}\, m_\infty(x_1)\,(x_1 - V_{\mathrm{ca}}).
\end{align*}
The parameters are chosen in the following way
\begin{align*}
    V_k &= -84, & V_l &= -60, & V_{\mathrm{Ca}} &= 120, &V_1 &= -1.2, \\
    V_2 &= 18, &g_k &= 8, & g_l &= 2, & c &= 20.\\
    V_3 &= 12, &V_4 &= 17.4, &g_{ca} &= 4, &g_{kca} &= 0.25\\
    y_0 &= 10, & \phi &= 0.23, &\mu &= 0.2, &\varepsilon &= 0.05,
\end{align*} and we consider the initial values \( (x_1, x_2,y)(0) = (-20,0,6)\).
The variable \(x_1\) corresponds to the membrane potential difference $V$ between the inside and the outside of the nerve cell, while $x_2$ is a recovery variable $w$ and $y$ denotes the calcium concentration [Ca] near the cell membrane. The parameter $I$, representing the external input current, is sampled from the interval $[44.8,45.5]$. For this range and the parameter values specified above, solutions converge to stable periodic bursting orbits. Within one global period, the trajectory exhibits predominantly 3,4, or 5 fast oscillations, depending on the value of $I$, see Figure~\ref{fig: number_of_bursts}. 
As $I$ increases toward 45.6, the number of fast oscillations per burst generally increases. In this regime, however, the spike count per period no longer varies smoothly with $I$, but exhibits irregular behavior. For $I > 45.6$, the system enters a regime where the classical fast-slow bursting structure is replaced by oscillatory dynamics dominated by the fast variables.

\paragraph{Fréchet Mean for Homogeneous Burst Count}
In the following, we compute the geometrical and dynamical Fréchet mean considering the homogeneous case, where all sample curves exhibit 4 bursts per period. We first compute the Fréchet mean using the framework introduced above. We consider $N=50$ closed sample curves resulting from parameter samples $I_1,\dots,I_N \in [45,45.3]$. Figure~\ref{fig:ml_homogeneous_b} displays the sample trajectories together with the Fréchet mean. The trajectories exhibit a similar overall structure, and the mean curve captures the characteristic pattern of four bursting oscillations within one period. To better illustrate the structure of the trajectories, we examine the time projections of the variables $V$ and $[\mathrm{Ca}]$, see Figure~\ref{fig:ml_homogeneous_a}. In each component, the Fréchet mean reproduces the characteristic pattern of four bursts with their typical shape. It also attains an intermediate length in each dimension, reflected in the endpoints of the curves. Thus, even for this complex three-dimensional bursting structure, the method yields a geometrically consistent representative trajectory across all variables.
As in the previous examples, we combine the geometric Fréchet mean with averaged dynamics to obtain a dynamically reconstructed representative trajectory. The associated speed profiles $v(s)$ and cumulative traversal times $\theta(s)$ are shown in Figure~\ref{fig:ml_homogeneous_c} together with the corresponding sample trajectories. Both quantities exhibit a coherent structure across the family of trajectories, allowing for a meaningful summary through the dynamic Fréchet mean. In particular, the averaged speed profile captures the characteristic alternation between slow and fast motion along the bursting orbit, resulting in a sensible geometric and dynamical summary of the family of trajectories through the Fréchet mean.

\begin{figure}[h]
\centering

\begin{subfigure}[t]{0.26\textwidth}
    \vspace{0.6cm}
    \centering

    \includegraphics[
        width=\linewidth,
        height=.18\textheight,
        keepaspectratio]{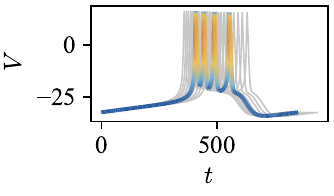}

    \vspace{0pt}

    \includegraphics[
        width=\linewidth,
        height=.18\textheight,
        keepaspectratio]{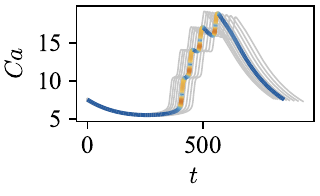}
    \vspace{0pt}
    \caption{Time projections}
    \label{fig:ml_homogeneous_a}
\end{subfigure}
\hspace{0.02\textwidth}
%
\begin{subfigure}[t]{0.38\textwidth}
    \vspace{0pt}
    \centering

    \includegraphics[width=\linewidth]{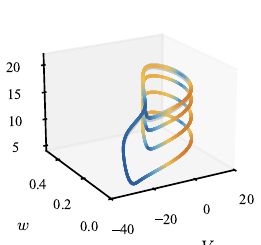}
    \vspace{0pt}
    \caption{Fréchet mean}
    \label{fig:ml_homogeneous_b}
\end{subfigure}
\hspace{0.02\textwidth}
%
\begin{subfigure}[t]{0.26\textwidth}
    \vspace{0.6cm}
    \centering

    \includegraphics[
        width=\linewidth,
        height=.18\textheight,
        keepaspectratio]{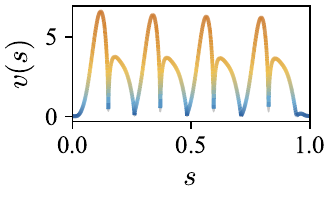}

    \vspace{0pt}

    \hspace*{-1.3em}
    \includegraphics[
        width=1.1\linewidth,
        height=.18\textheight,
        keepaspectratio]{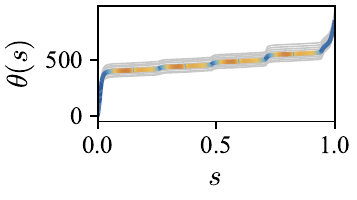}
        \vspace{-0.1cm}
    \caption{Dynamic profile}
    \label{fig:ml_homogeneous_c}
\end{subfigure}

\caption{Fr\'echet mean and reconstructed dynamics for the homogeneous Morris--Lecar bursting regime using trajectories generated from input currents $I\in[45,45.3]$. In this parameter regime, all sample trajectories exhibit four bursts within one global period. (a) Time projections of the membrane potential $V$ and calcium concentration $[\mathrm{Ca}]$ for the sample trajectories together with the reconstructed Fr\'echet mean. The mean reproduces the characteristic four-burst structure and captures representative amplitudes and temporal behavior across both state variables. (b) Geometric and dynamically reconstructed Fr\'echet mean in phase space. The resulting mean trajectory preserves the characteristic four-burst geometry of the sample together with the associated fast--slow bursting dynamics, indicating that both the geometric and temporal structures are represented consistently. (c) Aligned speed profiles $v(s)$ and cumulative traversal-time maps $\theta(s)$ for the sample trajectories and their reconstructed averages. The comparatively low variability across the sample reflects the homogeneous nature of the bursting regime and results in a stable and physically meaningful dynamic reconstruction that preserves the characteristic alternation between slow and fast motion.}

\label{fig:ml_homogeneous}
\end{figure}

\paragraph{Fréchet Mean for Heterogeneous Burst Count}
We now consider samples whose trajectories exhibit different numbers of bursts within a single period, leading to fundamentally distinct geometric structures. Constructing a meaningful average for such heterogeneous objects is inherently challenging, and we examine how the Fréchet mean behaves in this setting. Consider $N = 50$ trajectories obtained from random parameter samples $I_1, \dots, I_N \in [44.8, 45.4]$. The resulting distribution of burst counts per period is summarized in Table~\ref{tab: spike_distribution}, ranging from three to six bursts.
\begin{table}[t]
    \centering
    \caption{Burst count distribution for $ N = 50 $ Morris-Lecar trajectories with $ I \in [44.8,45.4] $. }
    \begin{tabular}{lcccc}
        \toprule
        Burst count $n$ & 3 & 4 & 5 & 6 \\
        \midrule
        Number of curves with $n$ bursts & 13 & 32 & 4 & 1 \\
        \bottomrule
    \end{tabular}
    
    \label{tab: spike_distribution}
\end{table}
We compute the Fréchet mean for this setting, see Figure~\ref{fig:ml_heterogeneous_b}. Although all sample trajectories exhibit fast-slow bursting dynamics, the number of bursts per period differs across the sample. In contrast to the homogeneous case, the Fréchet mean no longer produces a coherent bursting pattern. In the fast oscillatory regime, the bursts appear distorted or partially collapsed rather than fully developed. In the slow return phase, the agreement improves but remains noticeably weaker than in the structurally homogeneous setting. A clear visualization of this phenomenon is provided by the one-dimensional time projections of the variables $V$ and $\mathrm{Ca}$, see Figure~\ref{fig:ml_heterogeneous_a}. The misalignment of bursts becomes particularly evident in these projections. Although the Fréchet mean roughly exhibits four bursts in the variable $V$,  these bursts are visibly distorted and fail to reproduce the characteristic shape of the individual sample trajectories. This phenomenon is further reflected in the averaged speed profile shown in Figure~\ref{fig:ml_heterogeneous_c}. Because the speed maxima of the individual trajectories are shifted relative to one another, the averaging procedure blends these peaks rather than preserving their original structure. The resulting mean speed profile contains substantially more local maxima than the sample trajectories, leading to an artificial oscillatory behavior that is neither physically meaningful nor representative of the observed dynamics. Overall, the resulting mean trajectory does not correspond to a physically meaningful solution of the system. The breakdown arises from a structural incompatibility. The Fréchet mean continuously aligns and averages the sample curves, whereas the underlying variability is discrete. This structural mismatch prevents the averaged trajectory from preserving a coherent bursting structure. 

This heterogeneous setting raises two natural questions: How can one systematically quantify whether a computed Fréchet mean provides a meaningful summary of the sample trajectories? And how can one handle situations in which the averaging procedure becomes unreliable? In the following, we address the first question by applying the diagnostic tools introduced in Section~\ref{sec:diagnostics}.

\begin{figure}[h]
\centering

\begin{subfigure}[t]{0.26\textwidth}
    \vspace{0.7cm}
    \centering

    \includegraphics[
        width=\linewidth,
        height=.18\textheight,
        keepaspectratio]{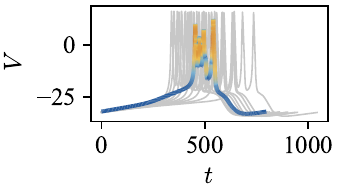}

    \vspace{0pt}

    \includegraphics[
        width=\linewidth,
        height=.18\textheight,
        keepaspectratio]{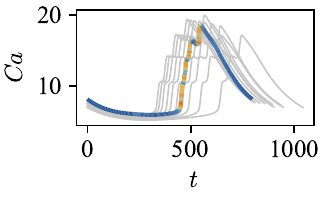}
    \vspace{0pt}
    \caption{Time projections}
    \label{fig:ml_heterogeneous_a}
\end{subfigure}
\hspace{0.02\textwidth}
%
\begin{subfigure}[t]{0.38\textwidth}
    \vspace{0pt}
    \centering

    \includegraphics[width=\linewidth]{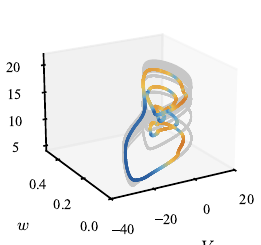}
    \vspace{0pt}
    \caption{Fréchet mean}
    \label{fig:ml_heterogeneous_b}
\end{subfigure}
\hspace{0.02\textwidth}
%
\begin{subfigure}[t]{0.26\textwidth}
    \vspace{0.7cm}
    \centering

    \includegraphics[
        width=\linewidth,
        height=.18\textheight,
        keepaspectratio]{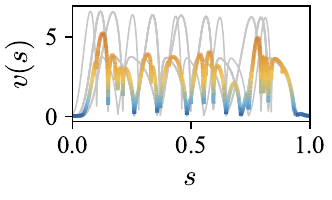}

    \vspace{0pt}

    \hspace*{-1.3em}
    \includegraphics[
        width=1.1\linewidth,
        height=.18\textheight,
        keepaspectratio]{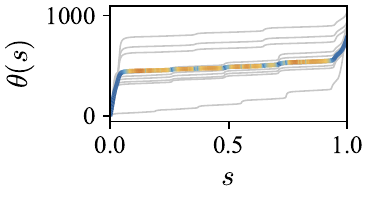}
        \vspace{-0.1cm}
    \caption{Dynamic profile}
    \label{fig:ml_heterogeneous_c}
\end{subfigure}

\caption{Fr\'echet mean and reconstructed dynamics for the heterogeneous Morris--Lecar bursting regime using trajectories generated from input currents $I\in[44.8,45.4]$. In this parameter regime, the sample trajectories exhibit different numbers of bursts within one global period, leading to substantial geometric heterogeneity. (a) Time projections of the membrane potential $V$ and calcium concentration $[\mathrm{Ca}]$ for the sample trajectories together with the reconstructed Fr\'echet mean. Although the mean roughly reflects the dominant four-burst structure present in the dataset, the bursts become visibly distorted due to averaging across incompatible bursting patterns. (b) Geometric and dynamically reconstructed Fr\'echet mean in phase space. The geometric averaging procedure fails to preserve a coherent bursting structure, producing partially collapsed and distorted oscillatory features caused by interpolation between trajectories with different burst counts. While the reconstructed dynamics still recover the distinction between slow motion along the slow manifold and fast burst segments, additional slow regions appear within the bursts themselves, indicating inconsistencies introduced by the averaging process. (c) Aligned speed profiles $v(s)$ and cumulative traversal-time maps $\theta(s)$. In particular, the speed profiles reveal multiple competing peaks originating from trajectories with different bursting structures. The averaging procedure attempts to combine these incompatible temporal features, resulting in a distorted dynamic summary with multiple broadened speed peaks. Overall, the figure demonstrates that the Fr\'echet mean no longer provides a geometrically or dynamically representative summary in the heterogeneous bursting regime.}

\label{fig:ml_heterogeneous}
\end{figure}

\subsection{Quantification of Fréchet Mean for Homogeneous and Heterogeneous Burst Count}
We now seek quantitative diagnostics that detect the structural breakdown. To this end, we employ curvature-based measures and medoid-based structural indices, introduced in Section~\ref{sec:diagnostics}, in order to quantify the sensibility of the Fréchet mean.

\paragraph{Curvature-based diagnostics}
We first analyze curvature-based quantities to compare the homogeneous and heterogeneous bursting regimes. Specifically, we consider the total curvature and the bending energy which measure cumulative turning and smoothness of a closed curve $\gamma$, respectively. In the homogeneous bursting regime, both quantities exhibit a tight concentration across the sample trajectories, and the corresponding values of the Fréchet mean lie within the same range, see Figure~\ref{fig:curvature_metrics_homogeneous}. This indicates that the mean preserves the characteristic geometric complexity and smoothness of the dataset. In contrast, the heterogeneous bursting regime reveals two distinct effects, see Figure~\ref{fig:curvature_metrics_heterogeneous}. First, the total curvature values form well-separated clusters. Since total curvature accumulates turning along the trajectory, these clusters correspond to trajectories with different numbers of bursts. Thus, total curvature naturally separates the dataset into geometrically distinct families. The Fréchet mean does not lie within any of these clusters, indicating that it does not share the structural characteristics of any coherent burst-count class. Second, while the bending energy of the sample trajectories remains relatively concentrated, the bending energy of the Fréchet mean is significantly larger. This reflects the presence of additional curvature oscillations introduced by averaging trajectories with incompatible burst structure. Hence, the mean exhibits artificial geometric irregularities not present in the data. Together, these observations show that total curvature detects geometric structures in the data, while bending energy reveals additional oscillations introduced by averaging incompatible trajectories, thereby providing diagnostics to assess the geometric representativeness of the Fréchet mean.

\begin{figure}[t]
\centering

\begin{subfigure}[t]{0.46\textwidth}
    \centering

    \includegraphics[width=0.48\linewidth]
    {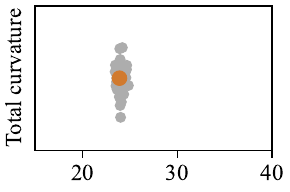}
    \hfill
    \includegraphics[width=0.48\linewidth]
    {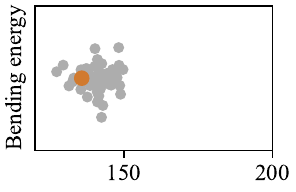}

    \caption{Homogeneous}
    \label{fig:curvature_metrics_homogeneous}
\end{subfigure}
\hfill
\begin{subfigure}[t]{0.46\textwidth}
    \centering

    \includegraphics[width=0.48\linewidth]
    {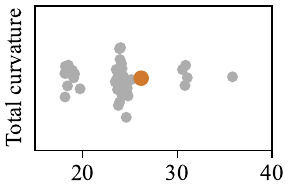}
    \hfill
    \includegraphics[width=0.48\linewidth]
    {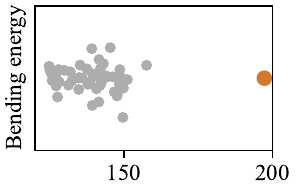}

    \caption{Heterogeneous}
    \label{fig:curvature_metrics_heterogeneous}
\end{subfigure}

\caption{Curvature-based diagnostics for the Morris--Lecar system in homogeneous and heterogeneous bursting regimes. (a) Homogeneous regime with parameter values $I\in[45,45.3]$. The total curvature and bending energy of the Fr\'echet mean lie within the range of the sample trajectories, which form a geometrically coherent data cloud. This indicates that the Fr\'echet mean preserves the characteristic geometric structure and curvature properties of the dataset when all trajectories exhibit the same burst count. (b) Heterogeneous regime with parameter values $I\in[44.8,45.4]$. The total curvature of the sample trajectories forms distinct clusters corresponding to different burst counts, while the Fr\'echet mean attains an intermediate curvature value located between these clusters. At the same time, the bending energies of the sample trajectories remain comparatively concentrated, whereas the Fr\'echet mean exhibits a substantially larger bending energy. This indicates the introduction of artificial geometric irregularities and oscillatory artifacts caused by averaging trajectories with incompatible bursting structures, thereby signaling a loss of structural representativeness of the Fr\'echet mean.}
\label{fig:curvature_metrics}
\end{figure}

\paragraph{Medoid-based diagnostics}
While curvature metrics detect geometric distortion and clustering at the level of individual trajectories, they do not directly quantify how well the Fréchet mean represents the dataset as a whole. To assess representativeness more directly, we compare the Fréchet mean to the Fréchet medoid and analyze complementary diagnostics, see Table~\ref{tab:medoid_diagnostics}.

\begin{table}[t]
\centering
\caption{Medoid-based diagnostic for homogeneous and heterogeneous bursting cases.}
\label{tab:medoid_diagnostics}
\begin{tabular}{lc|cccccc}
\toprule
 & $I$ & $\mathrm{Var}_F$ & $\mathrm{Var}_F^{\mathrm{med}}$ & $d(\mu_F, \mu_F^{\mathrm{med}})^2$ & $\delta_F$ & $\delta_F^{\mathrm{med}}$ & $\mathcal{S}$ \\
\midrule
Homogeneous   & 45--45.3 & 0.030 & 0.032 & 0.006 & 0.005 & 0.0002 & 32.220 \\
Heterogeneous & 44.8--45.4 & 58.280 & 80.008 & 21.885 & 21.885 & 0.002 & 10199.766 \\
\bottomrule
\end{tabular}
\end{table}

The empirical Fréchet variance $\mathrm{Var}_F$ increases from $0.030$ in the homogeneous case to $58.280$ in the heterogeneous case. A similar increase is observed for the medoid variance $\mathrm{Var}_F^{\mathrm{med}}$. This substantial growth reflects a dramatic increase in global geometric dispersion of the dataset. However, variance alone does not detect structural incompatibility: large dispersion may arise either from continuous deformation within a coherent family of trajectories or from the coexistence of structurally distinct regimes. Thus, variance quantifies global spread but does not by itself imply a structural breakdown of the Fréchet mean.

The squared distance $d(\mu_F,\mu_F^{\mathrm{med}})^2$ measures the deviation of the Fréchet mean from the medoid and thus quantifies how far the mean lies from a structurally central observed trajectory. It increases from $0.006$ in the homogeneous regime to $21.885$ in the heterogeneous regime. In the homogeneous case, the Fréchet mean lies very close to the medoid, indicating that it behaves like a typical trajectory of the dataset. In contrast, in the heterogeneous regime, the mean is far from the medoid, see Figure~\ref{fig: ml-medoid-vs-mean}, suggesting that it no longer reflects a representative observed orbit. 

\begin{figure}[h]
\centering

\begin{subfigure}[b]{0.37\linewidth}
    \centering
    \includegraphics[width=\linewidth]{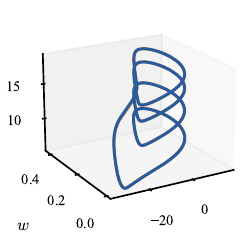}
    \caption{Homogeneous}
\end{subfigure}
\hspace{0.05\textwidth}
\begin{subfigure}[b]{0.37\linewidth}
    \centering
    \includegraphics[width=\linewidth]{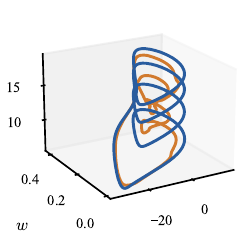}
    \caption{Heterogeneous}
\end{subfigure}

\caption{Comparison of the Fr\'echet mean and Fr\'echet medoid for the homogeneous and heterogeneous Morris--Lecar bursting regimes. (a) In the homogeneous regime, the Fr\'echet mean lies very close to the Fr\'echet medoid and reproduces the characteristic four-burst structure of the sample trajectories. This indicates that the Fr\'echet mean remains geometrically representative of the dataset and closely resembles a typical observed trajectory. (b) In the heterogeneous regime, the Fr\'echet mean differs substantially from the Fr\'echet medoid and exhibits distorted oscillatory structures that are not present in the observed data. The strong discrepancy between mean and medoid indicates a loss of representativeness caused by structural heterogeneity and averaging across trajectories with incompatible bursting patterns.}
\label{fig: ml-medoid-vs-mean}
\end{figure}

To compare the geometric position of the Fréchet mean and medoid relative to the data, we consider the nearest-neighbor distances \(\delta_F \) and \(\delta_F^{\mathrm{med}}\) and the resulting ratio \(\mathcal{S} \). The quantity $\delta_F$ measures how close the Fréchet mean lies to its nearest observed trajectory, while $\delta_F^{\mathrm{med}}$ measures the distance from the medoid to its closest neighboring sample curve. The ratio $\mathcal{S}$ thus quantifies how strongly the mean is locally supported by the data relative to a structurally central observed trajectory. In the homogeneous regime, both distances are small ($0.005$ and $0.002$), yielding $\mathcal{S}=32.220$. Thus, although the mean is not itself a data point, it lies on the same geometric scale as intrinsic nearest-neighbor distances in the dataset. In the heterogeneous regime, however, a clear scale separation emerges: while the medoid remains close to at least one neighboring trajectory ($\delta_F^{\mathrm{med}}=0.002$), the nearest-neighbor distance of the Fréchet mean increases dramatically to $21.885$, yielding a ratio of $\mathcal{S} \approx 10200$. This indicates that the mean no longer lies within the local support scale of the data. In other words, although a representative trajectory exists within the dataset in the form of the medoid, the Fréchet mean becomes geometrically isolated and therefore fails to represent any typical observed trajectory.

\paragraph{Summary of diagnostics}
Taken together, the diagnostics show that the Fréchet mean no longer represents a typical trajectory when multiple burst-count classes coexist. Rather than preserving a coherent trajectory structure, it averages across incompatible bursting regimes and introduces geometric features that are not present in the observed data. Consequently, these diagnostics provide practical indicators for identifying situations in which the Fréchet mean ceases to serve as a meaningful summary of the dataset.

\section{Discussion and Outlook}\label{sec:discussion}
\subsection{Summary and Main Contributions}
This work developed a framework for computing Fréchet means of periodic trajectories by modeling
closed curves as elements of the quotient space $\mathcal{C} = L^2(\mathcal{S}^1, \mathbb{R}^d)/\mathcal{S}^1$,
equipped with a phase-aligned metric obtained by minimizing the $L^2$ distance over circular phase
shifts. A central observation is that the choice of parametrization fundamentally determines the
character of the resulting mean: time parametrization preserves dynamical information but may distort
the geometric shape of the orbit family, while arc length parametrization yields geometrically
faithful averages at the cost of discarding temporal dynamics. To reconcile these two viewpoints,
geometry and dynamics are treated as complementary components. The geometric Fréchet mean is
computed from arc length parametrized curves, and representative dynamics are subsequently
reconstructed via harmonic averaging of aligned speed profiles, motivated by the additivity of
traversal times along the mean geometry. A diagnostic framework based on curvature quantities and
medoid-based statistics was introduced to assess when the empirical Fréchet mean constitutes a
meaningful representative of the sample. Applications to the Van der Pol oscillator, the
Rosenzweig--MacArthur predator--prey model, and the Morris--Lecar bursting model showed that the
framework yields geometrically consistent and dynamically meaningful summaries in structurally
homogeneous settings, while the diagnostics reliably detect cases where qualitative heterogeneity in the sample causes the empirical Fréchet mean to lose its interpretability as a representative trajectory.

The main contributions are as follows. A metric framework for periodic trajectories
is introduced by adapting the quotient space construction from shape analysis to the
dynamical systems setting, deliberately restricting the quotient structure to circular phase shifts only and thus preserving the physical
meaning carried by the absolute position, orientation, and scale of a trajectory.
This yields $\mathcal{C} = L^2(\mathcal{S}^1, \mathbb{R}^d)/\mathcal{S}^1$ as the natural setting for
comparing periodic orbits independently of their starting point. Existence of the
empirical Fréchet mean in $\mathcal{C}$ is then established in
Proposition~\ref{prop:existence minimizer} via the direct method in the calculus of
variations, adapted to this infinite-dimensional quotient setting. Theorem~\ref{theo:convergence}
analyzes the iterative alignment-averaging algorithm as block coordinate descent on
the joint functional $G(c, \tau)$, establishing monotone decrease of the empirical
Fréchet functional and subsequential weak
convergence of the iterates in $L^2$.
The decoupled geometric and dynamic harmonic averaging procedure constitutes a principled method for constructing a
representative mean trajectory that combines geometric fidelity with a consistent
dynamical interpretation. The diagnostic framework combines curvature-based and medoid-based quantities to
assess representativeness, providing practical tools for detecting geometric
artifacts, displacement from the data cloud, and lack of local support as distinct
failure modes of the empirical Fréchet mean.


\subsection{Limitations}
A limitation of the proposed framework is its reliance on the $L^2$ metric together with alignment by constant circular shifts. While computationally efficient, this alignment can only correct global phase differences and cannot account for local timing variability within a trajectory. As a result, corresponding features that are slightly shifted in time across trajectories may be averaged imperfectly, leading to smoothing artifacts such as broadened or attenuated maxima in the empirical Fréchet mean. More flexible alignment methods, such as those based on the square-root velocity function (SRVF) framework~\cite{srivastava2016functional}, can accommodate local temporal deformations and may reduce such artifacts, but they come at substantially higher computational cost, do not preserve the Hilbert space structure underlying the present approach, and introduce additional questions regarding the physical interpretability of the induced time warping.

The convergence result of Theorem~\ref{theo:convergence} is weaker than one
might hope: the non-convexity of $F^{(N)}$ implies that the algorithm may converge to different
limit points depending on the initialization, and whether these limit points are stationary points of $F^{(N)}$ is not characterized. Nevertheless, the theorem guarantees monotone decrease of the empirical Fréchet functional and boundedness of the iterates, providing a basic level of algorithmic stability. While stronger convergence results remain open, these properties are sufficient for the practical computations considered in this work.


Moreover, the empirical Fréchet mean is a curve constructed to minimize a geometric functional, not a trajectory of the underlying dynamical system. It need not satisfy any differential equation, respect invariant manifold structure, or correspond to a physically realizable orbit. Rather than representing a solution of the system, its role is to provide a notion of central tendency for a family of trajectories in the chosen metric space. As such, its relevance derives from its ability to characterize the orbit family as a whole, rather than from satisfying the dynamical laws that generated the individual trajectories.



\subsection{Future Directions}

\paragraph{Uncertainty quantification for the Fr\'echet mean}
The empirical Fr\'echet mean is a point estimate and the present work provides no
quantification of how the estimator varies with the sample, for instance as $N$
grows or under resampling. 
A bootstrap procedure, e.g. resampling the orbit family with replacement and
recomputing the Fr\'echet mean for each resample, would yield pointwise
confidence tubes around the mean curve and bootstrap distributions of the
scalar diagnostic quantities. The latter is particularly useful, as
substantial variability of the diagnostics across resamples would indicate
that the failure diagnosis itself is unstable. For finite-dimensional
Riemannian manifolds, \cite{bhattacharya2005large} establishes bootstrap
confidence regions for the Fr\'echet mean with coverage error $O_p(n^{-2})$. This would provide a principled notion of statistical confidence for Fr\'echet means and their diagnostic quantities in dynamical systems applications.

\paragraph{Handling structural failure}
The diagnostic framework identifies when the Fréchet mean fails to represent the sample, but does
not prescribe a remedy. When structural heterogeneity is detected, a natural approach is to
partition the orbit family into coherent subgroups using $k$-medoids clustering with the pairwise
distances $d(\gamma_j, \gamma_k)$ already available from the algorithm, compute class-conditional
Fréchet means within each group, and use the diagnostics to guide the choice of the number of
classes. This connects the framework to mixture models on metric spaces and makes the failure
response operational.

\paragraph{Connection to bifurcation theory}
The Morris--Lecar example illustrates that topological changes in the orbit family---here, 
changes in burst count per period---cause the isolation ratio and variance gap to increase
sharply. Tracking these diagnostic quantities along a parameter path therefore provides a
data-driven signal for dynamical regime change, without requiring knowledge of the underlying
bifurcation structure. A precise formulation of this connection, for instance relating the growth
of $\mathcal{S}$ to specific bifurcation types, is a natural direction for future work. 

\paragraph{Alternative metric and quotient structures}
A systematic comparison of the Fréchet means obtained under different metric and alignment choices---including the $L^2$ metric with constant circular shifts, Sobolev-type metrics incorporating derivative information, and SRVF-based approaches with dynamic warping---would help clarify which geometric structures are most appropriate for different classes of dynamical systems. Such a study could assess the tradeoff between computational tractability, geometric fidelity, and dynamical interpretability. 

\paragraph{Alternative notions of mean}
The Fréchet mean minimizes the expected squared distance and is the natural $L^2$ notion of central tendency on a metric space, but alternative notions may be more appropriate in certain settings. The geometric median, minimizing $\mathbb{E}[d(c,X)]$, is more robust to geometrically extreme trajectories and requires only a finite first moment. The Doss expectation~\cite{doss1949moyenne} similarly avoids second-moment assumptions, while the Sturm barycenter~\cite{sturm2003probability} provides an iterative construction of the Fréchet mean in metric spaces of nonpositive curvature. A systematic comparison of these notions could clarify how different definitions of centrality affect robustness, representativeness, and computational tractability for trajectory data.

\paragraph{Extensions to other invariant sets and applications}
The framework is currently restricted to periodic orbits. A natural extension covers quasi-periodic
trajectories on invariant tori, where the relevant quotient group is $\mathbb{T}^2$ rather than
$\mathcal{S}^1$. Applications to ergodic dynamics on strange attractors would require a different notion of
representative geometry altogether. On the applied side, the framework is directly motivated by
uncertainty quantification for oscillatory climate patterns such as ENSO cycles, where families of
near-periodic orbits arise naturally from parameter uncertainty or model ensembles, and where a
principled notion of a representative cycle is of direct practical interest.

\section*{Acknowledgments}
Generative AI tools were used during the preparation of this manuscript to assist with discussion of research ideas, software development, computational implementation, and refinement of exposition and phrasing. All mathematical results, computational experiments, and conclusions were verified by the authors. The authors assume responsibility for all content.


\printbibliography

\end{document}